\documentclass[11pt,oneside,english]{amsart}
\usepackage[T1]{fontenc}
\usepackage[latin9]{inputenc}
\usepackage{geometry}
\usepackage{babel}
\usepackage{mathrsfs}
\usepackage{amstext}
\usepackage{amsthm}
\usepackage{amssymb}
\usepackage[unicode=true,pdfusetitle,
 bookmarks=true,bookmarksnumbered=false,bookmarksopen=false,
 breaklinks=false,pdfborder={0 0 1},backref=false,colorlinks=false]
 {hyperref}

\makeatletter
\numberwithin{equation}{section}
\numberwithin{figure}{section}
\theoremstyle{plain}
\newtheorem{thm}{\protect\theoremname}
\theoremstyle{plain}
\newtheorem{conjecture}[thm]{\protect\conjecturename}
\theoremstyle{plain}
\newtheorem{cor}[thm]{\protect\corollaryname}
\theoremstyle{definition}
\newtheorem{defn}[thm]{\protect\definitionname}
\theoremstyle{plain}
\newtheorem{prop}[thm]{\protect\propositionname}

\makeatother

\providecommand{\conjecturename}{Conjecture}
\providecommand{\corollaryname}{Corollary}
\providecommand{\definitionname}{Definition}
\providecommand{\propositionname}{Proposition}
\providecommand{\theoremname}{Theorem}

\begin{document}
\pagestyle{plain}
\title{averages of diagonal elliott-halberstam problem twisted by m\"obius
function with sobolev and h\"older-zygmund weights}
\author{marco cantarini}
\address{Department of Mathematics and Computer Science, Via Vanvitelli 1,
06123 Perugia (PG), Italy.}
\email{marco.cantarini@unipg.it}
\begin{abstract}
Recalling that the so-called Elliott-Halberstam conjecture twisted
by the M\"obius function $\mu(n)$ claims that

\[
\sum_{q\leq N^{\theta}}\max_{y\leq N}\max_{(a,q)=1}\left|\sum_{\underset{{\scriptstyle n\equiv a\,\mod\,q}}{n\leq y}}\Lambda(n)\mu\left(N-n\right)-\frac{1}{\varphi\left(q\right)}\sum_{n\leq y}\Lambda(n)\mu\left(N-n\right)\right|\ll\frac{N}{\log\left(N\right)^{A}}
\]
for every $A>0$, where $0<\theta<1$ is fixed, and also recalling
that the validity of this conjecture, in combination with the validity
of the classical Elliott-Halberstam for suitable $\theta$, proves
the binary Goldbach conjecture, in this paper we study weighted average
variants of this problem. We will show that, under Generalized Riemann
Hypothesis, a weak version of the Gonek-Hejhal conjecture and working
with weights belonging to the Sobolev space $W^{2,1}$ or in the H\"older-Zygmund
spaces $\mathcal{C}^{\delta}$ for suitable range of $\delta$, the
bound of the average is consistent with the bound of the ``diagonal
versions'' of this conjecture (that is, taking $y=N$ and taking
$n\equiv N\mod q)$. In particular, in the case of weights in Sobolev
space, the consistent upper bound holds for the whole $0<\theta<1$
and, in the case of weights in the H\"older-Zygmund class $\mathcal{C}^{\delta}$,
for $\theta$ that depends on the choice of $\delta$ but still not
below the $1/2-2\varepsilon$ threshold. 
\end{abstract}

\maketitle

\section{introduction}

\subsection{Context}

The binary Goldbach conjecture is one of the oldest problems in mathematics
and it states that every even integer greater than 2 is the sum of
two primes. Despite its elementary formulation, it has proven incredibly
difficult to attack, due to its deep connections to some of the most
complex problems in analytic number theory. Consequently, many strategies
have been proposed to study this problem; a nice survey of these ideas
can be found \cite{L}. 

However, these attempts have proven to be completely unsuitable for
a definitive answer to the main conjecture and it is, in fact, a common
opinion that there are still serious obstacles to overcome that current
tools seem not to be able to handle. A widely held view is that the
main obstruction is the so-called parity problem. Roughly speaking,
sieve methods are extremely effective at controlling the distribution
of almost-primes, but they do not by themselves distinguish reliably
between integers with an even number of prime factors and integers
with an odd number of prime factors. This limitation is a serious
obstacle in problems where the exact multiplicative structure matters
(for a more detailed analysis, see \cite{FI 2010}, chapt. 16). A
classical illustration is provided by Chen's theorem \cite{CHE},
which states that every sufficiently large even integer is the sum
of a prime and an integer with at most two prime factors; the theorem
comes remarkably close to Goldbach's problem, but the remaining gap
reflects precisely the parity barrier. A similar theorem to Chen's
result can be found in \cite{CHE} regarding the twin prime conjecture
with the same limitations; this is not surprising, since these two
conjectures are widely believed to be closely connected, and one may
reasonably expect that a strategy proving one of them would also shed
light on the other.

The parity problem is closely related to the behavior of oscillatory
multiplicative functions, especially the M\"obius function
\[
\mu\left(n\right):=\begin{cases}
1, & n=1\\
\left(-1\right)^{k}, & n=p_{1}\cdots p_{k},\,k\in\mathbb{N},\,p_{j}\text{ distinct prime numbers}\\
0, & \text{otherwise.}
\end{cases}
\]

A useful heuristic principle, often called the M\"obius randomness
principle, is that $\mu(n)$ should exhibit substantial cancellation
when correlated with sufficiently structured sequences. In informal
terms, for a ``reasonable'' complex sequence $a_{n}$, the twisted
sum $\sum_{n\leq x}\mu(n)a_{n}$ is expected to be much smaller than
its trivial bound because of cancellation (see \cite{IK2004}, chap.
13, for more details).

Many attempts have been made to overcome the parity barrier, see e.g.
\cite{FI1998,P2015}. In 2017 Murty and Vatwani \cite{MV} postulated
a new conjecture that breaks down such obstacle, connecting the Elliott-Halberstam
conjecture and the equidistribution of the M\"obius function on shifted
primes. We begin by recalling the classical Elliott-Halberstam conjecture.
\begin{conjecture}
\label{conj:EH classic}(Elliott-Halberstam conjecture EH$\left(N^{\theta}\log\left(N\right)^{C}\right)$)
Let $0<\theta<1$ and $C>0$ be fixed. Then, for every $A>0$ we have
\[
\sum_{q\leq N^{\theta}\log\left(N\right)^{C}}\max_{y\leq N}\max_{(a,q)=1}\left|\sum_{\underset{{\scriptstyle n\equiv a\,\mod\,q}}{n\leq y}}\Lambda(n)-\frac{y}{\varphi\left(q\right)}\right|\ll_{A}\frac{N}{\log\left(N\right)^{A}}
\]
for every sufficiently large natural number $N$, where $\Lambda(n)$
is the Von Mangoldt function and $\varphi(n)$ is the Euler totient
function.
\end{conjecture}

Clearly, if $\theta<1/2$, it is, up to the usual technical normalization
of the level of distribution, the well-known Bombieri-Vinogradov theorem
(see, e.g., \cite{IK2004}, chap. 17). In that sense, Bombieri-Vinogradov
may be viewed as a ``satisfactory substitute'', citing \cite{IK2004},
of the GRH, which is, today, out of reach. Murty and Vatwani proposed
a version of the conjecture that connect EH and the distribution of
the M\"obius function.
\begin{conjecture}
(Shifted M\"obius Elliott--Halberstam Conjecture EH$_{\mu_{h}}\left(N^{\theta}\right)$)
Let $0<\theta<1$ be fixed. Then, for any $A>0$, we have
\[
\sum_{q\leq N^{\theta}}\max_{y<N}\max_{(a,q)=1}\left|\sum_{\underset{{\scriptstyle n\equiv a\,\mod\,q}}{n\leq y}}\Lambda(n)\mu\left(n+h\right)-\frac{1}{\varphi\left(q\right)}\sum_{n\leq y}\Lambda(n)\mu\left(n+h\right)\right|\ll_{A}\frac{N}{\log\left(N\right)^{A}}
\]
for every sufficiently large natural number $N$.
\end{conjecture}

One of the main results in \cite{MV} is that a suitable combination
of the classical Elliott-Halberstam conjecture and its shifted M\"obius-twisted
analogue would break the parity barrier in a strong quantitative form
for the twin primes problem. Note that, in what follows, $p$ will
always denote a prime number.
\begin{thm}
\label{thm:1.1 mur}(Theorem $1.1$ of \cite{MV}) Let $h\neq0$ be
a fixed even integer. Suppose that the conjectures EH$\left(N^{\theta}\log\left(N\right)^{C}\right)$
and EH$_{\mu_{h}}\left(N^{1-\theta}\right)$ are true for some fixed
$\theta<1$ and a suitably large fixed $C$. Then, for all sufficiently
large positive integers $N,$we then obtain the following:

1) We have
\[
\sum_{n\leq N}\Lambda(n)\Lambda(n+h)\geq\left(1+o\left(1\right)\right)\mathfrak{S_{t}}\left(h\right)\left(1-\mathcal{A}_{h}\right)N
\]
where
\[
\mathcal{A}_{h}:=\prod_{\underset{{\scriptstyle p>2}}{p\nmid h}}\left(1-\frac{1}{p\left(p-1\right)}\right)
\]
and
\[
\mathfrak{S}_{t}\left(h\right):=\prod_{p\mid h}\left(1+\frac{1}{p-1}\right)\prod_{p\nmid h}\left(1-\frac{1}{\left(p-1\right)^{2}}\right)
\]
is the singular series for the twin primes conjecture.

2) The asymptotic formula
\[
\sum_{n\leq N}\Lambda(n)\Lambda(n+h)\sim\mathfrak{S}_{t}\left(h\right)N
\]
is equivalent to the condition
\[
\sum_{n\leq N}\Lambda(n)\mu(n+h)=o\left(N\right).
\]
\end{thm}

Note that, by the Bombieri-Vinogradov theorem, the previous result
implies that the proof of the twin prime conjecture follows from the
proof of the EH$_{\mu_{h}}\left(N^{\theta}\right)$ conjecture for
$\theta>1/2$.

The same ideas apply to Goldbach's problem. Indeed, Huang and Li proposed
the following twisted version with the M\"obius version of the EH
conjecture (see \cite{HL2020}).
\begin{conjecture}
\label{conj:twistEH}(Elliott-Halberstam conjecture twisted by M\"obius
function EH$_{\mu}$$\left(N^{\theta}\right)$) Let $0<\theta<1$
be fixed. Then, for every $A>0$ we have
\[
\sum_{q\leq N^{\theta}}\max_{y<N}\max_{(a,q)=1}\left|\sum_{\underset{{\scriptstyle n\equiv a\,\mod\,q}}{n\leq y}}\Lambda(n)\mu\left(N-n\right)-\frac{1}{\varphi\left(q\right)}\sum_{n\leq y}\Lambda(n)\mu\left(N-n\right)\right|\ll_{A}\frac{N}{\log\left(N\right)^{A}}
\]
for every sufficiently large natural number $N$.
\end{conjecture}

From these conjectures, the mentioned authors proved the following
theorem (see \cite{HL2020}).
\begin{thm}
\label{thm:GB proof}For a fixed $A>0$ assume that EH$\left(N^{\theta}\log\left(N\right)^{2A+8}\right)$
and EH$_{\mu}\left(N^{1-\theta}\right)$ hold for some $0<\theta<1$.
Then, for all sufficiently large positive integers $N,$the following
hold:

1) We have

\[
R\left(N\right):=\sum_{n<N}\Lambda\left(n\right)\Lambda\left(N-n\right)\geq N\mathfrak{S}\left(N\right)\left(1-\mathcal{A}_{N}\right)+O_{A,\theta}\left(\frac{N}{\log\left(N\right)^{A}}\right)
\]
where 
\[
\mathfrak{S}\left(N\right):=\begin{cases}
2\prod_{\underset{{\scriptstyle p>2}}{p\mid N}}\left(1+\frac{1}{p-2}\right)\prod_{p>2}\left(1-\frac{1}{\left(p-1\right)^{2}}\right), & N\text{ even}\\
0, & N\text{ odd}
\end{cases}
\]
 is the singular series for the binary Goldbach conjecture and $\mathcal{A}_{N}$
is the same of Theorem \ref{thm:1.1 mur}.

2) The assertions

\[
R\left(N\right)\sim N\mathfrak{S}\left(N\right)
\]
 and 
\[
\sum_{n<N}\Lambda(n)\mu\left(N-n\right)=o\left(N\right)
\]
 are equivalent.
\end{thm}

As above, once the classical Elliott-Halberstam input is supplied
by Bombieri-Vinogradov, proving EH$_{\mu}(N^{\theta})$ for some $\theta>1/2$
would yield the binary Goldbach conjecture for all sufficiently large
even integers. The work of Huang and Li may also be viewed as a continuation
of earlier work of Pan \cite{P1982}, who identified a related obstruction
through the analysis of a truncated remainder term of the form
\[
R^{*}\left(N\right):=\sum_{n<N}\left(\sum_{\underset{{\scriptstyle d_{1}>Q}}{d_{1}\mid n}}\mu\left(d_{1}\right)\log\left(d_{1}\right)\right)\left(\sum_{\underset{{\scriptstyle d_{2}>Q}}{\underset{{\scriptstyle (d_{2},N)=1}}{d_{2}\mid(N-n)}}}\mu\left(d_{2}\right)\log\left(d_{2}\right)\right)
\]
where the truncation at $Q:=N^{1/2}\left(\log\left(N\right)\right)^{-20}$
depends on the Bombieri-Vinogradov theorem.

\subsection{Aims of the paper and methods of proof}

The main goals of this paper are to study weighted average variants
of EH$_{\mu}\left(N^{\theta}\right)$, that is
\begin{equation}
\sum_{n}\sum_{m}\Lambda\left(n\right)\chi\left(n\right)\mu\left(m\right)f\left(\frac{n+m}{N}\right),\,\sum_{n}\sum_{m}\Lambda\left(n\right)\chi\left(n\right)\mu\left(m\right)\log\left(m\right)f\left(\frac{n+m}{N}\right)\label{eq:main average}
\end{equation}
where $N\geq4$ is a natural number, $f$ is a suitable weight and
$\chi$ is a primitive non-principal Dirichlet character $\mod q$,
where $q>1$ is an integer. These quantities are natural because,
using orthogonality of Dirichlet characters and making simple considerations,
one allows to rewrite congruence sums, always assuming $(a,q)=1$
and $y<N$, such as
\[
\sum_{\underset{{\scriptstyle n\equiv a\,\mod\,q}}{n\leq y}}\Lambda(n)\mu\left(N-n\right)-\frac{1}{\varphi\left(q\right)}\sum_{n\leq y}\Lambda(n)\mu\left(N-n\right)
\]
\[
=\frac{1}{\varphi\left(q\right)}\sum_{\chi\neq\chi_{0}}\overline{\chi}\left(a\right)\sum_{n\leq y}\Lambda\left(n\right)\chi^{*}\left(n\right)\mu\left(N-n\right)+O\left(\omega\left(q\right)\log\left(N\right)\right)
\]
where $\omega\left(q\right):=\sum_{p\mid q}1\ll\log\left(q\right)$
and $\chi^{*}$ is the primitive character that induces $\chi$, so
the average (\ref{eq:main average}) is what we need, paying a reasonable
error. We actually consider variants of EH$_{\mu}\left(N^{\theta}\right)$
in order to restore a useful symmetry between the prime variable and
the M\"obius variable: unlike the original form of EH$_{\mu}(N^{\theta})$,
where the summation is cut off at $y$ but the twist is evaluated
at $N-n$, the weighted expression depends symmetrically on the combination
$n+m$. This symmetry is an important feature of our method. Moreover,
in \cite{HL2020} (but actually also in \cite{MV}, since the proofs
are very similar) it is crucial to have $\max_{y\leq N}\sum_{n\leq y}$
because it is used to derive a logarithmically weighted variant, namely
a bound for the same discrepancy with $\mu(N-n)\log(N-n)$ in place
of $\mu(N-n)$. Thus, one cannot simply discard the $\max_{y\leq N}$,
unlike $\max_{(a,q)=1}$ that actually can be removed (see \cite{HL2020,MV}),
without replacing it by a different and suitably robust hypothesis.
For this reason, we propose the following modified, and in some sense
weaker, version of EH$_{\mu}\left(N^{\theta}\right)$.
\begin{conjecture}
\label{conj:DiagonalEH}(Diagonal Elliott-Halberstam conjecture twisted
by M\"obius function and logarithm dEH$_{\mu,\log}$$\left(N^{\theta}\right)$).
Let $0<\theta<1$ be fixed. Then, for every $A>0$ we have
\[
\sum_{\underset{{\scriptstyle \left(N,q\right)=1}}{q\leq N^{\theta}}}\left|\sum_{\underset{{\scriptstyle n\equiv N\,\mod\,q}}{n<N}}\Lambda(n)g\left(N-n\right)-\frac{1}{\varphi\left(q\right)}\sum_{n<N}\Lambda(n)g\left(N-n\right)\right|\ll_{A}\frac{N}{\log\left(N\right)^{A}}
\]
for every sufficiently large natural number $N$, where $g(n)\equiv\mu(n)$
or $g(n)\equiv\mu(n)\log\left(n\right).$
\end{conjecture}

This conjecture removes the two maxima appearing in EH$_{\mu}\left(N^{\theta}\right)$,
but in exchange it requires the same strength of cancellation both
in the pure M\"obius case and in the logarithmically weighted case.

However, in Section $2$, we show that if we replace Conjecture \ref{conj:twistEH}
with Conjecture \ref{conj:DiagonalEH}, then Theorem \ref{thm:GB proof}
still holds, and this fact justifies our study of (\ref{eq:start}).
In the same section we also show that the maximum over the residue
class can be removed from the classical Elliott-Halberstam input as
well.

After that, we consider (\ref{eq:start}) with different types of
weights: if $f$ is sufficiently regular, that is, it has its support
contained in $[\alpha,\beta)$ and its restriction $f_{\vert\left(\alpha,\beta\right)}$belongs
to the Sobolev space $W^{2,1}\left(\alpha,\beta\right)$, then we
prove, under GRH and assuming the simplicity of the non-trivial zeros
of the Riemann zeta function $\zeta(s)$, a truncated explicit formula
for (\ref{eq:main average}). In somewhat informal terms, the main
tool for proving this explicit formula will be a version of Abel's
summation formula in two dimensions that allows us to \textquotedbl decouple\textquotedbl{}
the explicit formulas of the averages of the arithmetic functions
that define the problem. This explicit formula implies an averaged
estimate compatible with the expected order of magnitude in this averaged
setting, of the shape
\begin{equation}
\sum_{\underset{{\scriptstyle \left(N,q\right)=1}}{1<q\leq N^{\theta}}}\frac{1}{\varphi\left(q\right)}\left|\sum_{\chi\neq\chi_{0}}\overline{\chi}\left(N\right)\sum_{N\alpha<n\leq N\beta}\sum_{m\leq N\beta-n}\Lambda\left(n\right)\chi\left(n\right)\mu\left(m\right)f\left(\frac{n+m}{N}\right)\right|\label{eq:first result}
\end{equation}
\[
\ll_{\varepsilon}N^{2-\varepsilon}E\left(f^{\prime\prime}\right)
\]
where $E\left(f^{\prime\prime}\right)$ is an (explicit) error that
depends only on $f^{\prime\prime}$, for every $\theta<1-2\varepsilon$
and where $\varepsilon>0$ is a sufficiently small number. The same
result holds in the logarithmically weighted case. 

An important remark is that in (\ref{eq:first result}), and during
the rest of the paper, we commit the following abuse of notation for
the sake of greater readability: the sum $\sum_{\chi\neq\chi_{0}}$
denotes the sum over all non-principal characters modulo $q$, while
$\chi(n)$ is understood as the value of the primitive character inducing
$\chi$.

We also show that, taking a discrete version of the two dimensional
Abel formula, we are able to work with different weights $f$. Essentially,
the problem shifts from the analysis of integrals with second derivative
of the weight to sums with forward differences of order 2 of the weight.
In order to have a control over the growth of the forward differences
of $f$, it becomes natural to consider functions that lie in appropriate
H\"older-Zygmund spaces.

More precisely, we show that if $f$ is in the H\"older-Zygmund class
$\mathcal{C}^{\delta}\left(\mathbb{R}\right),1\leq\delta<2$, again
with the support of $f$ contained in $[\alpha,\beta)$, then we obtain
a truncated explicit formula and a corresponding averaged bound compatible
with the expected conjectural size, that is
\[
\sum_{\underset{{\scriptstyle \left(N,q\right)=1}}{1<q\leq N^{\theta}}}\frac{1}{\varphi\left(q\right)}\left|\sum_{\chi\neq\chi_{0}}\overline{\chi}\left(N\right)\sum_{N\alpha<n\leq N\beta}\sum_{m\leq N\beta-n}\Lambda\left(n\right)\chi\left(n\right)\mu\left(m\right)f\left(\frac{n+m}{N}\right)\right|\ll_{\varepsilon,\delta,\alpha,\beta}N^{2-\varepsilon}
\]
 for $\theta<\delta-1-2\varepsilon$ if $\delta\in\left[\frac{3}{2},2\right)$
and for $\theta<\frac{1}{2}-2\varepsilon$ if $\delta\in\left[1,3/2\right)$.
Finally, we apply the same strategy to the case with logarithmic weighting,
which allow us to show, in both the continuous and discrete cases,
results similar to the previous ones.

The paper is organized as follows. In Section $2$, we show that the
diagonal conjecture dEH$_{\mu,\log}$ is sufficient for the proof
of the Goldbach's problem. Section $3$ collects the preliminary definitions,
lemmas, and auxiliary tools used later on. Section $4$ contains the
first main results: we prove the truncated explicit formulae and the
upper bounds for the weighted average of dEH$_{\mu,\log}$ with $g(n)=\mu(n)$
in both cases $W^{2,1}$ and $\mathcal{C}^{\delta}$. In Section $5$
we prove similar results for the logarithmic case of dEH$_{\mu,\log}$
with $g(n)=\mu(n)\log(n)$. In Section $6$ we show that some interesting
examples of averages are a special cases of our results, such as,
for instance, the well-known Ces\`aro-Riesz averages.

\subsection{Notation}

We use $Y\ll Z$, $Y=O\left(Z\right)$ to denote the estimate $\left|Y\right|\leq CZ$
for some suitable $C>0$. If the constant $C$ depends on some parameters
$A,B,C\dots$ we write $Y\ll_{A,B,C,\dots}Z$ or $Y=O_{A,B,C,\dots}\left(Z\right)$.
We write $Y\asymp Z$ if $Y\ll Z\ll Y$. 

We write $\sum_{\chi\neq\chi_{0}}$ with the meaning of the sum on
all the Dirichlet characters $\mod q$ excluding the principal one
but, with an abuse of notation, $\chi\left(n\right)$ is always a
primitive character that induces $\chi$. We denote
\[
\sum_{n\leq A}:=\sum_{0<n\leq A}
\]
and in all other cases it will be explicitly written from which integer
the sum starts.

With the symbol $C^{n}\left(A\right),\,n\in\mathbb{N}$, where $A=\mathbb{R}$
of $A$ is an interval, we indicate the space of continuous function
whose derivatives up to the $n$-th order are continuous in $A$.
Clearly, $C^{0}\left(A\right)$ is the space of the continuous functions
in $A$. With $AC\left(A\right)$ we indicate the space of the absolutely
continuous functions in $A$.

With $L^{\infty}\left(\mathbb{R}\right)$ we indicate the space of
measurable functions in $\mathbb{R}$ whose essential supremum is
finite, equipped with the norm
\[
\left\Vert f\right\Vert _{\infty}:=\underset{{\scriptstyle x\in\mathbb{R}}}{\text{ess sup}}\left|f\left(x\right)\right|.
\]

With $\Gamma\left(z\right)$ we denote the Euler Gamma function
\[
\Gamma\left(z\right):=\int_{0}^{+\infty}x^{z-1}e^{-x}dx,\,\text{Re}(z)>0,
\]
with $B\left(z,h\right)$ the Beta function
\[
B\left(z,h\right):=\int_{0}^{1}x^{z-1}\left(1-x\right)^{h-1}dx=\frac{\Gamma\left(z\right)\Gamma\left(h\right)}{\Gamma\left(z+h\right)},\,\text{Re}(z)>0,\text{Re}(h)>0
\]
and with $\psi^{\left(0\right)}\left(z\right)$ the Digamma function
\[
\psi^{\left(0\right)}\left(z\right):=\frac{\Gamma^{\prime}\left(z\right)}{\Gamma\left(z\right)},z\in\mathbb{C},\,z\neq0,-1,-2,\dots.
\]

\section{proof of goldbach's conjecture assuming conjecture \ref{conj:DiagonalEH}}

In this section we prove a version of Theorem \ref{thm:GB proof}
with Conjecture \ref{conj:DiagonalEH}. We will often recall some
results in \cite{HL2020} because the proof is, essentially, the same. 
\begin{thm}
\label{thm:GB proof DIAG}Let $0<\theta<1$ and $C>0$ be fixed. Assume
that
\begin{equation}
\sum_{\underset{{\scriptstyle \left(N,q\right)=1}}{q\leq N^{\theta}\log\left(N\right)^{C}}}\left|\sum_{\underset{{\scriptstyle n\equiv N\,\mod\,q}}{n<N}}\Lambda(n)-\frac{N}{\varphi\left(q\right)}\right|\ll_{A}\frac{N}{\log\left(N\right)^{A}}\label{eq:EH senza max-1}
\end{equation}
for every sufficiently large natural number $N$ and dEH$_{\mu,\log}\left(N^{1-\theta}\right)$
holds for the same $\theta$. Then, using the same notations of Theorem
\ref{thm:GB proof}, we have
\[
R\left(N\right):=\sum_{n<N}\Lambda\left(n\right)\Lambda\left(N-n\right)\geq N\mathfrak{S}\left(N\right)\left(1-\mathcal{A}_{N}\right)+O_{A,\theta}\left(\frac{N}{\log\left(N\right)^{A}}\right).
\]
Moreover, the assertions $R\left(N\right)\sim N\mathfrak{S}\left(N\right)$
and $\sum_{n<N}\Lambda(n)\mu\left(N-n\right)=o\left(N\right)$ are
equivalent.
\end{thm}

\begin{proof}
Fix $0<\theta<1$ and let $\alpha=N^{\theta}.$ Following \cite{HL2020},
we write
\[
R\left(N\right)=S_{1}\left(\alpha\right)+S_{2}\left(\alpha\right)+O\left(N^{1/2}\log\left(N\right)^{3}\right)
\]
and
\[
S_{1}\left(\alpha\right):=\sum_{n<N}\Lambda\left(n\right)\mu\left(N-n\right)^{2}\sum_{\underset{{\scriptstyle d\leq\alpha}}{d\mid(N-n)}}\mu\left(d\right)\log\left(\frac{1}{d}\right)
\]
\[
S_{2}\left(\alpha\right):=\sum_{n<N}\Lambda\left(n\right)\mu\left(N-n\right)^{2}\sum_{\underset{{\scriptstyle d>\alpha}}{d\mid(N-n)}}\mu\left(d\right)\log\left(\frac{1}{d}\right).
\]
Let us study $S_{1}\left(\alpha\right)$. From the identity
\[
\mu\left(N-n\right)^{2}=\sum_{b^{2}\vert(N-n)}\mu\left(b\right)
\]
we have
\[
S_{1}\left(\alpha\right)=\sum_{d\leq\alpha}\mu\left(d\right)\log\left(\frac{1}{d}\right)\sum_{\underset{{\scriptstyle \left[b^{2},d\right]<N}}{b<\sqrt{N}}}\mu\left(b\right)\sum_{\underset{{\scriptstyle n\equiv N\mod\left(\left[b^{2},d\right]\right)}}{n<N}}\Lambda\left(n\right).
\]
Following section $3.1$ of \cite{HL2020} we, split the sum to $b>B:=\log\left(N\right)^{A+4}$and
$b\leq B$. This lead to
\[
S_{1}\left(\alpha\right)=N\sum_{d\leq\alpha}\frac{\mu\left(d\right)\log\left(1/d\right)}{\varphi\left(d\right)}\sum_{\underset{{\scriptstyle \left(bd,N\right)=1}}{b\leq B}}\frac{\mu\left(b\right)\varphi\left(\left(b,d\right)\right)}{b\varphi\left(b\right)}
\]
\[
+E_{1}\left(\alpha\right)+O_{A}\left(\frac{N}{\log\left(N\right)^{A}}\right),
\]
where
\[
E_{1}\left(\alpha\right)\ll\log\left(N\right)\sum_{z\leq\alpha B^{2}}\tau_{3}\left(z\right)\left|\sum_{\underset{{\scriptstyle n\equiv N\mod z}}{n<N}}\Lambda\left(n\right)-\frac{N}{\varphi\left(z\right)}\right|
\]
\[
\ll N^{1/2}\log\left(N\right)^{3/2}\left(\sum_{z\leq\alpha B^{2}}\frac{\tau_{3}\left(z\right)^{2}}{z}\right)^{1/2}\left(\sum_{z\leq\alpha B^{2}}\left|\sum_{\underset{{\scriptstyle n\equiv N\mod z}}{n<N}}\Lambda\left(n\right)-\frac{N}{\varphi\left(z\right)}\right|\right)^{1/2}
\]
by Cauchy-Schwarz and from the trivial bound
\[
\left|\sum_{\underset{{\scriptstyle n\equiv N\mod z}}{n<N}}\Lambda\left(n\right)-\frac{N}{\varphi\left(z\right)}\right|\ll\frac{N\log\left(N\right)}{z},
\]
 where
\[
\tau_{3}\left(z\right):=\sum_{abc=z}1.
\]
Since, by partial summation, we have
\[
\sum_{z\leq\alpha B^{2}}\frac{\tau_{3}\left(z\right)^{2}}{z}\ll\log\left(N\right)^{9}
\]
(see section $3.1$ in \cite{HL2020}, p. $338$ for more details)
it remains to evaluate 
\[
\sum_{z\leq\alpha B^{2}}\left|\sum_{\underset{{\scriptstyle n\equiv N\mod z}}{n<N}}\Lambda\left(n\right)-\frac{N}{\varphi\left(z\right)}\right|\ll_{A}\frac{N}{\log\left(N\right)^{2A+12}}
\]
by (\ref{eq:EH senza max-1}). So, from the standard estimate for
$\mathfrak{S}\left(N\right)$, we get
\[
S_{1}\left(\alpha\right)=N\mathfrak{S}\left(N\right)+O\left(\frac{N}{\log\left(N\right)^{A}}\right).
\]
Continuing to follow \cite{HL2020}, we split $S_{2}\left(\alpha\right)$
in two parts
\[
S_{2}\left(\alpha\right)=S_{3}\left(\alpha\right)-S_{4}\left(\alpha\right)+O\left(\frac{N}{\log\left(N\right)^{A}}\right)
\]
where
\[
S_{3}\left(\alpha\right):=\sum_{\underset{{\scriptstyle \left(k,N\right)=1}}{k<(N-1)/\alpha}}\mu\left(k\right)\log\left(k\right)\sum_{\underset{{\scriptstyle n\equiv N\mod k}}{n<N}}\Lambda\left(n\right)\mu\left(N-n\right)
\]
and 
\[
S_{4}\left(\alpha\right):=\sum_{\underset{{\scriptstyle \left(k,N\right)=1}}{k<(N-1)/\alpha}}\mu\left(k\right)\sum_{\underset{{\scriptstyle n\equiv N\mod k}}{n<N}}\Lambda\left(n\right)\mu\left(N-n\right)\log\left(N-n\right).
\]

We analyze $S_{3}\left(\alpha\right).$ We have
\[
S_{3}\left(\alpha\right)=\sum_{\underset{{\scriptstyle \left(k,N\right)=1}}{k<(N-1)/\alpha}}\frac{\mu\left(k\right)\log\left(k\right)}{\varphi\left(k\right)}\sum_{n<N}\Lambda\left(n\right)\mu\left(N-n\right)+E_{3}\left(\alpha\right)
\]
where
\[
E_{3}\left(\alpha\right):=\sum_{\underset{{\scriptstyle \left(k,N\right)=1}}{k<(N-1)/\alpha}}\mu\left(k\right)\log\left(k\right)\left(\sum_{\underset{{\scriptstyle n\equiv N\mod k}}{n<N}}\Lambda\left(n\right)\mu\left(N-n\right)-\frac{1}{\varphi\left(k\right)}\sum_{n<N}\Lambda\left(n\right)\mu\left(N-n\right)\right)
\]
(see \cite{HL2020}, p. $340$) hence by Conjecture \ref{conj:DiagonalEH}
with $g(n)=\mu\left(n\right)$ we get
\[
E_{3}\left(\alpha\right)\ll_{A}\frac{N}{\log\left(N\right)^{A}}
\]
and from the standard estimate
\[
\sum_{\underset{{\scriptstyle \left(k,N\right)=1}}{k<R}}\frac{\mu\left(k\right)\log\left(k\right)}{\varphi\left(k\right)}=-\mathfrak{S}\left(N\right)+O\left(e^{-C\sqrt{\log\left(R\right)}}\right)
\]
we conclude that 
\[
S_{3}\left(\alpha\right)=-\mathfrak{S}\left(N\right)\sum_{n<N}\Lambda\left(n\right)\mu\left(N-n\right)+O_{A}\left(\frac{N}{\log\left(N\right)^{A}}\right)
\]
(see \cite{HL2020}, p. $341$ for more details). With a similar argument
we have that 
\[
S_{4}\left(\alpha\right)=\sum_{\underset{{\scriptstyle \left(k,N\right)=1}}{k<(N-1)/\alpha}}\frac{\mu\left(k\right)}{\varphi\left(k\right)}\sum_{n<N}\Lambda\left(n\right)\mu\left(N-n\right)\log\left(N-n\right)+E_{4}\left(\alpha\right)
\]
where
\[
E_{4}\left(\alpha\right):=\sum_{\underset{{\scriptstyle \left(k,N\right)=1}}{k<(N-1)/\alpha}}\mu\left(k\right)\left(\sum_{\underset{{\scriptstyle n\equiv N\mod k}}{n<N}}\Lambda\left(n\right)\mu\left(N-n\right)\log\left(N-n\right)\right.
\]
\[
\left.-\frac{1}{\varphi\left(k\right)}\sum_{n<N}\Lambda\left(n\right)\mu\left(N-n\right)\log\left(N-n\right)\right)
\]
and again by Conjecture \ref{conj:DiagonalEH} with $g(n)=\mu\left(n\right)\log\left(n\right)$
we get
\[
E_{4}\left(\alpha\right)\ll_{A}\frac{N}{\log\left(N\right)^{A}}.
\]
 It remains to observe that
\[
\sum_{\underset{{\scriptstyle \left(k,N\right)=1}}{k<(N-1)/\alpha}}\frac{\mu\left(k\right)}{\varphi\left(k\right)}\ll e^{-C\sqrt{\log\left(\left(N-1\right)/\alpha\right)}}
\]
and
\[
\sum_{n<N}\Lambda\left(n\right)\mu\left(N-n\right)\log\left(N-n\right)\ll N\log\left(N\right)^{2}
\]
to get
\[
S_{4}\left(\alpha\right)\ll_{A}\frac{N}{\log\left(N\right)^{A}}
\]
(see \cite{HL2020}, p. $341,342$ for more details). The rest of
the proof is exactly the same of \cite{HL2020}.
\end{proof}
Therefore, using the Conjecture \ref{conj:DiagonalEH} allows us to
effectively estimate the error terms $E_{1}\left(\alpha\right),E_{3}\left(\alpha\right)$
and $E_{4}\left(\alpha\right)$ and thus prove the binary Goldbach's
conjecture in exactly the same way as done in \cite{HL2020}. 

\section{preliminary results, definitions and main tools}

In this section we have to recall several facts and prove some preliminary
results. For this reason, we divide this section further subsections.

\subsection{The main tools}

One of the main tools used to show our theorems is an identity proved
in \cite{CGZ2025}. It can be viewed as a two-dimensional Abel summation
formula. Versions of two-dimensional Abel summation do appear in the
literature (see, e.g, \cite{KT2006} or \cite{BP2018}), but the presented
formulation makes explicit the link between a weighted average involving
two arithmetic functions and the Laplace convolution of the corresponding
unweighted averages taken separately. This point of view is particularly
useful, since it allows us to combine explicit formulae for the individual
averages, whenever available, and thereby derive an explicit formula
for the original weighted average. To the best of our knowledge, Abel
summation has not previously been formulated in exactly this way in
the literature.
\begin{thm}
\label{thm:MAINCGZ}Let $g_{1},g_{2}$ be arithmetical functions,
$\eta\in\mathbb{R}^{+}$, $f:\mathbb{R}\rightarrow\mathbb{C}$, and
assume that:

1) $f$ has its support in $\left[\alpha,\beta\right),\,0\leq\alpha<\beta$,
and $\beta\in\mathbb{R}$

2) $f\in C^{1}\left(\alpha,\beta\right)$.

3) $f^{\prime}\in AC\left(\alpha,\beta\right)$.

Then, if 
\[
G_{j}\left(x\right):=\begin{cases}
\sum_{n\leq x}g_{j}\left(n\right), & x>0\\
0, & \text{otherwise}
\end{cases}
\]
for $j=1,2$, we get
\[
\sum_{\eta\alpha<n\leq\eta\beta}\sum_{m\leq\eta\beta-n}g_{2}\left(m\right)g_{1}\left(n\right)f\left(\frac{m+n}{\eta}\right)=G_{2}\left(\eta\alpha\right)\int_{\alpha}^{\beta}G_{1}\left(\eta v-\eta\alpha\right)f^{\prime}\left(v\right)dv
\]
\[
+\frac{1}{\eta}\int_{\alpha}^{\beta}f^{\prime\prime}\left(w\right)\int_{\eta\alpha}^{\eta w}G_{2}\left(s\right)G_{1}\left(\eta w-s\right)dsdw.
\]
\end{thm}

Observe that, essentially, we require that $f_{\vert\left(\alpha,\beta\right)}$
belongs to the Sobolev spaces $W^{2,1}\left(\alpha,\beta\right)$.
We recall the definition, just for completeness: if $I$ is an open
interval, then
\[
W^{1,p}\left(I\right):=\left\{ f\in L^{p}\left(I\right):\,\exists g\in L^{p}\left(I\right):\int_{I}fh^{\prime}=-\int_{I}gh,\,\forall h\in C_{c}^{1}\left(I\right)\right\} 
\]
where $C_{c}^{1}\left(I\right)$ are the function $C^{1}\left(I\right)$
with compact support in $I$, and
\[
W^{m,p}\left(I\right):=\left\{ f\in W^{m-1,p}\left(I\right):f^{\prime}\in W^{m-1,p}\left(I\right)\right\} ,\,m\in\mathbb{N},\,m\ge2
\]
(see, e.g., \cite{B2011}, chapt. $8)$. We recall that is possible
to extend the results of the previous theorem also in the case $\beta=+\infty$
(see \cite{CGZ2025}). Note that the hypotheses of the Theorem imply
that $f\left(\alpha^{+}\right),f^{\prime}\left(\alpha^{+}\right)$
exists and are finite (but not necessary equal to $0$) and $f\left(\beta^{-}\right)=f^{\prime}\left(\beta^{-}\right)=0$. 

In order to work with more general weights, we propose now a discrete
version of the previous theorem. Firstly, we recall the definition
of $r$-th order forward difference, with $r\in\mathbb{N}^{+}$, as
\[
\Delta_{h}^{r}\left(f,x\right):=\sum_{k=0}^{r}\binom{r}{k}\left(-1\right)^{r-k}f\left(x+hk\right),\,x,h\in\mathbb{R},
\]
\[
\Delta_{h}\left(f,x\right):=\Delta_{h}^{1}\left(f,x\right)
\]
and the property
\begin{equation}
\Delta_{h}^{r}=\Delta_{h}\left[\Delta_{h}^{r-1}\right]\label{eq:prod fow diff}
\end{equation}
(see. e.g., \cite{DVL1993}, chapt. $7$).
\begin{thm}
\label{thm:MainDiscrete}Let $g_{1},g_{2}$ be arithmetical functions,
let $f:\mathbb{R}\rightarrow\mathbb{C}$ and assume that:

1) f has its support in $[\alpha,\beta)$, $0\leq\alpha<\beta$$.$ 

2) $G_{1}(1)=G_{1}(0)=G_{2}(0)=0.$

Then, taking $\eta\in\mathbb{R}^{+}$ such that $\eta\beta\in\mathbb{N}^{+},\,\eta\alpha\in\left\{ 0,1\right\} $
we have
\[
\sum_{\eta\alpha<n\leq\eta\beta}\sum_{m\leq\eta\beta-n}g_{2}\left(m\right)g_{1}\left(n\right)f\left(\frac{m+n}{\eta}\right)=\sum_{k=2}^{\eta\beta-1}\left(\sum_{u=0}^{k}G_{1}\left(u\right)G_{2}\left(k-u\right)\right)\Delta_{1/\eta}^{2}\left(f,\frac{k}{\eta}\right).
\]
\end{thm}

\begin{proof}
By classical summation by parts and since $f(\beta)=f\left(\beta+\frac{1}{\eta}\right)=0$,
because $f$ has its support in $[\alpha,\beta)$, we have
\[
\sum_{m\leq\eta\beta-n}g_{2}\left(m\right)f\left(\frac{m+n}{\eta}\right)=-\sum_{m\leq\eta\beta-n-1}G_{2}(m)\left[f\left(\frac{m+n+1}{\eta}\right)-f\left(\frac{m+n}{\eta}\right)\right]
\]
and so defining 
\[
C(n):=-\sum_{m\leq\eta\beta-n-1}G_{2}(m)\left[f\left(\frac{m+n+1}{\eta}\right)-f\left(\frac{m+n}{\eta}\right)\right]
\]
we have, again by partial summation
\[
\sum_{\eta\alpha<n\leq\eta\beta}g_{1}\left(n\right)\sum_{m\leq\eta\beta-n}g_{2}\left(m\right)f\left(\frac{m+n}{\eta}\right)=G_{1}\left(\eta\beta\right)C\left(\eta\beta\right)-G_{1}\left(\eta\alpha\right)C\left(\eta\alpha+1\right)
\]
\[
-\sum_{\eta\alpha<n\leq\eta\beta-1}G_{1}\left(n\right)\left(C\left(n+1\right)-C\left(n\right)\right).
\]
Now clearly
\[
C\left(\eta\beta\right)=G_{1}\left(\eta\alpha\right)=0
\]
hence
\[
\sum_{\eta\alpha<n\leq\eta\beta}g_{1}\left(n\right)\sum_{m\leq\eta\beta-n}g_{2}\left(m\right)f\left(\frac{m+n}{\eta}\right)=-\sum_{\eta\alpha<n\leq\eta\beta-1}G_{1}\left(n\right)\left(C\left(n+1\right)-C\left(n\right)\right)
\]
\[
=\sum_{\eta\alpha<n\leq\eta\beta-1}G_{1}\left(n\right)\sum_{m\leq\eta\beta-n-2}G_{2}(m)\left[f\left(\frac{m+n+2}{\eta}\right)-f\left(\frac{m+n+1}{\eta}\right)\right]
\]
\[
-\sum_{\eta\alpha<n\leq\eta\beta-1}G_{1}\left(n\right)\sum_{m\leq\eta\beta-n-1}G_{2}(m)\left[f\left(\frac{m+n+1}{\eta}\right)-f\left(\frac{m+n}{\eta}\right)\right]
\]
\[
=\sum_{\eta\alpha<n\leq\eta\beta-1}G_{1}\left(n\right)\sum_{m\leq\eta\beta-n-1}G_{2}(m)\Delta_{1/\eta}^{2}\left(f,\frac{m+n}{\eta}\right)
\]
then, taking $k=m+n$ and using condition $2)$, we obtain
\[
\sum_{\eta\alpha<n\leq\eta\beta}g_{1}\left(n\right)\sum_{m\leq\eta\beta-n}g_{2}\left(m\right)f\left(\frac{m+n}{\eta}\right)=\sum_{\eta\alpha<k\leq\eta\beta-1}\Delta_{1/\eta}^{2}\left(f,\frac{k}{\eta}\right)\sum_{u=0}^{k}\left(G_{1}(u)G_{2}(k-u)\right)
\]
as wanted.
\end{proof}
The previous result can be proved in a more general settings, but
paying the price of having an identity with a more complicated form.
\begin{cor}
\label{cor:discrete}Under the previous hypotheses, we have
\[
\sum_{\eta\alpha<n\leq\eta\beta}g_{1}\left(n\right)\sum_{m\leq\eta\beta-n}g_{2}\left(m\right)f\left(\frac{m+n}{\eta}\right)=\sum_{\eta\alpha<k\leq\eta\beta-1}\Delta_{1/\eta}^{2}\left(f,\frac{k}{\eta}\right)\int_{0}^{k}G_{1}(s)G_{2}(k-s)ds
\]
\[
+\sum_{\eta\alpha<k\leq\eta\beta-1}\Delta_{1/\eta}^{2}\left(f,\frac{k}{\eta}\right)\sum_{u=0}^{k}G_{1}(u)g_{2}(k-u).
\]
\end{cor}

\begin{proof}
It is enough to observe that
\[
\int_{0}^{k}G_{1}(s)G_{2}(k-s)ds=\sum_{u=0}^{k-1}\int_{u}^{u+1}G_{1}(s)G_{2}(k-s)ds=\sum_{u=0}^{k-1}G_{1}(u)G_{2}(k-u-1)
\]
hence
\[
\sum_{u=0}^{k}G_{1}(u)G_{2}(k-u)-\sum_{u=0}^{k-1}G_{1}(u)G_{2}(k-u-1)
\]
\[
=\sum_{u=0}^{k-1}G_{1}(u)\left[G_{2}(k-u)-G_{2}(k-u-1)\right]
\]
\[
=\sum_{u=0}^{k-1}G_{1}(u)g_{2}\left(k-u\right).
\]
\end{proof}

\subsection{H\"older-Zygmund spaces}

As mentioned in the introduction, we are interested in weights defined
by functions belonging to suitable function spaces. In the previous
subsection we considered the case in which the restriction $f_{\vert\left(\alpha,\beta\right)}\in W^{2,1}\left(\alpha,\beta\right)$.
In this subsection, we turn to the second setting under consideration,
namely H\"older-Zygmund spaces. We take this class because the previous
theorems show that we need an effective way to control finite differences
of the weight. We therefore recall the definition of these spaces
and some of their basic properties. Everything will be formulated
on $\mathbb{R},$ since our functions are defined on $\mathbb{R}$,
even when their support is contained in an interval; this is important
because the behavior near the boundary is different.

\begin{defn}
(H\"older-Zygmund space) Let $\delta>0$. We define the H\"older-Zygmund
space as
\[
\mathcal{C}^{\delta}\left(\mathbb{R}\right):=\left\{ f\in L^{\infty}\left(\mathbb{R}\right)\cap C^{0}\left(\mathbb{R}\right):\left\Vert f\right\Vert _{\mathcal{C}^{\delta}\left(\mathbb{R}\right)}<+\infty\right\} 
\]
where
\[
\left\Vert f\right\Vert _{\mathcal{C}^{\delta}\left(\mathbb{R}\right)}:=\left\Vert f\right\Vert _{\infty}+\sup_{x,h\in\mathbb{R},\,h\neq0}\frac{\left|\Delta_{h}^{1}\left(f,x\right)\right|}{h^{\delta}}
\]
if $0<\delta<1$,
\[
\left\Vert f\right\Vert _{\mathcal{C}^{1}\left(\mathbb{R}\right)}:=\left\Vert f\right\Vert _{\infty}+\sup_{x,h\in\mathbb{R},\,h\neq0}\frac{\left|\Delta_{h}^{2}\left(f,x\right)\right|}{h}
\]
if $\delta=1$. If $\delta>1$, say $m<\delta\leq m+1,\,m\in\mathbb{N}$,
then we define the H\"older-Zygmund space as
\[
\mathcal{C}^{\delta}\left(\mathbb{R}\right):=\left\{ f\in L^{\infty}\left(\mathbb{R}\right)\cap C^{m}\left(\mathbb{R}\right):\left\Vert f\right\Vert _{\mathcal{C}^{\delta}\left(\mathbb{R}\right)}<+\infty\right\} 
\]
where the norm is defined recursively as
\[
\left\Vert f\right\Vert _{\mathcal{C}^{\delta}\left(\mathbb{R}\right)}=\left\Vert f\right\Vert _{\mathcal{C}^{\delta-1}\left(\mathbb{R}\right)}+\left\Vert f^{\prime}\right\Vert _{\mathcal{C}^{\delta-1}\left(\mathbb{R}\right)}.
\]
\end{defn}

(see, e.g, \cite{R2022}) Note that $\mathcal{C}^{1}\left(\mathbb{R}\right)$
is also called Zygmund spaces, and if $\delta\notin\mathbb{N}$, $\mathcal{C}^{\delta}\left(\mathbb{R}\right)$
is equivalent to the classical H\"older spaces. Moreover, we have
the following result.
\begin{thm}
\label{thm:finitediff bound}If $f\in L^{\infty}\left(\mathbb{R}\right)\cap C^{0}\left(\mathbb{R}\right)$
and $0<\delta<n,\,n\in\mathbb{N}$, then $f\in\mathcal{C}^{\delta}\left(\mathbb{R}\right)\Leftrightarrow\left|\Delta_{h}^{n}\left(f,x\right)\right|\leq C\left|h\right|^{\delta}$
for all $x,h\in\mathbb{R}$.
\end{thm}

For a reference, see \cite{K1983}, Theorem 6.1. We also recall that
the following useful relations for function in Zygumd and H\"older--Zygmund
class about the definite difference $\Delta_{h}^{1}\left(f,x\right)$;
indeed, we have
\begin{equation}
f\in\mathcal{C}^{1}\left(\mathbb{R}\right)\Rightarrow\left|\Delta_{h}^{1}\left(f,x\right)\right|\leq M\left|h\log\left(\frac{1}{h}\right)\right|,\,M>0\label{eq:LogLip}
\end{equation}
for all $x,h\in\mathbb{R},\,h\neq0$ (see \cite{C1995}, Prop. 2.3.7)
and, if $\delta>1$
\[
f\in\mathcal{C}^{\delta}\left(\mathbb{R}\right)\Rightarrow\left|\Delta_{h}^{1}\left(f,x\right)\right|\leq M\left|h\right|,\,M>0
\]
that is, $f$ verifies the classical Lipschitz condition, for all
$x,h\in\mathbb{R}$, since
\[
\mathcal{C}^{\delta}\left(\mathbb{R}\right)\hookrightarrow\text{Lip}\left(\mathbb{R}\right)
\]
if $\delta>1$, where 
\[
\text{Lip}\left(\mathbb{R}\right):=\left\{ f\in L^{\infty}\left(\mathbb{R}\right):\left\Vert f\right\Vert _{\text{Lip}\left(\mathbb{R}\right)}<+\infty\right\} 
\]
and
\[
\left\Vert f\right\Vert _{\text{Lip}\left(\mathbb{R}\right)}=\left\Vert f\right\Vert _{\infty}+\sup_{x,h\in\mathbb{R},h\neq0}\frac{\left|\Delta_{h}^{1}\left(f,x\right)\right|}{\left|h\right|}
\]
(see again \cite{R2022}). 

\subsection{Estimate for ratio of Gamma function $\Gamma(x)$ and other results}

It is a standard fact that, when one considers a discrete convolution
of functions admitting an explicit formula involving series over zeros,
the corresponding explicit formula for the convolution is related,
in one way or another, to a sum involving a convolution of those zeros.
Very often, this leads to the need to control series involving the
Euler beta function or ratios of gamma functions. This phenomenon
also arises in our work, so it is convenient to record the following
proposition.
\begin{prop}
\label{prop:gamma est}Let $x,y\in\mathbb{R}$. Then
\[
\left|\frac{\Gamma\left(\frac{1}{2}+ix\right)\Gamma\left(\frac{1}{2}+iy\right)}{\Gamma\left(2+i(x+y)\right)}\right|\ll\begin{cases}
\frac{e^{\pi\left(\left|x+y\right|-\left|x\right|-\left|y\right|\right)/2}}{\left|x+y\right|^{3/2}}, & \left|x+y\right|>1\\
e^{-\pi\left(\left|x\right|+\left|y\right|\right)/2}, & \left|x+y\right|\leq1,
\end{cases}
\]
\end{prop}

\begin{proof}
From the well-known relations
\[
\left|\Gamma\left(\frac{1}{2}+ix\right)\right|^{2}=\frac{\pi}{\cosh\left(\pi x\right)},\,\left|\Gamma\left(2+i(x+y)\right)\right|^{2}=\left(1+\left(x+y\right)^{2}\right)\frac{\pi\left|x+y\right|}{\sinh\left(\pi\left|x+y\right|\right)}
\]
(see \cite{O2010}, $5.4.3$ and $5.4.4$), then
\[
\left|\frac{\Gamma\left(\frac{1}{2}+ix\right)\Gamma\left(\frac{1}{2}+iy\right)}{\Gamma\left(2+i(x+y)\right)}\right|^{2}=\pi\frac{\sinh\left(\pi\left|x+y\right|\right)}{\left(1+\left(x+y\right)^{2}\right)\left|x+y\right|\cosh\left(\pi x\right)\cosh\left(\pi y\right)}.
\]
Assume $\left|x+y\right|>1.$ Then, since
\[
\sinh\left(\pi\left|s\right|\right)\leq\frac{e^{\pi\left|s\right|}}{2},\cosh\left(\pi s\right)\geq\frac{e^{\pi\left|s\right|}}{2}
\]
valid for all $s\in\mathbb{R},$ we get
\[
\left|\frac{\Gamma\left(\frac{1}{2}+ix\right)\Gamma\left(\frac{1}{2}+iy\right)}{\Gamma\left(2+i(x+y)\right)}\right|\ll\frac{e^{\pi\left(\left|x+y\right|-\left|x\right|-\left|y\right|\right)/2}}{\left|x+y\right|^{3/2}}.
\]
If $\left|x+y\right|\leq1$, since
\[
\frac{\sinh\left(\pi\left|s\right|\right)}{\pi\left|s\right|}\ll1
\]
then we have
\[
\left|\frac{\Gamma\left(\frac{1}{2}+ix\right)\Gamma\left(\frac{1}{2}+iy\right)}{\Gamma\left(2+i(x+y)\right)}\right|\ll e^{-\pi\left(\left|x\right|+\left|y\right|\right)/2}.
\]
\end{proof}
Convergence of double series on the non-trivial zeros is a delicate
matter; for this reason, in \cite{CGZ2026} , we prove a theorem about
the convergence of a double series over the non-trival zeros of the
Riemann zeta function $\zeta(s)$ and a suitable function $f(\rho)$
that verifies some hypotheses. We need a similar results but now for
the non-trivial zeros of $L\left(s,\chi\right)$. We first recall
the following result.
\begin{prop}
Let $T\geq4$ be a real number and $N\left(T\right)$ be the number
of non-trivial zeros $\rho=\sigma+i\gamma$ of $\zeta\left(s\right)$
in the rectangle $0<\sigma<1$, $0<\gamma<T$. Moreover, let $\chi$
a primitive character $\mod q$ with $q\in\mathbb{N}^{+},q>1$ and
let $N\left(T,\chi\right)$ be the number of non-trivial zeros $\rho_{\chi}=\sigma_{\chi}+i\gamma_{\chi}$
in the rectangle $0<\sigma_{\chi}<1,0\leq\gamma_{\chi}\leq T$. Then,
we have
\begin{equation}
N\left(T\right)=\frac{T}{2\pi}\log\left(\frac{T}{2\pi}\right)-\frac{T}{2\pi}+O\left(\log\left(T\right)\right),\label{eq:N(T)}
\end{equation}
\begin{equation}
N\left(T,\chi\right)=\frac{T}{2\pi}\log\left(\frac{qT}{2\pi}\right)-\frac{T}{2\pi}+O\left(\log\left(qT\right)\right).\label{eq:N(T,=00005Cchi)}
\end{equation}
\end{prop}

See, e.g, \cite{DAVEN}, chapt. 15 and 16 or \cite{MV2006}, chapt.
14. Clearly, from the previous formulae we may derive also the cases
$\left|\gamma\right|\leq T,\left|\gamma_{\chi}\right|\leq T$, due
to the symmetry of the zeros of $\zeta\left(s\right)$ and from the
fact that the number of zeros of $L\left(s,\chi\right)$ with $-T\leq\gamma_{\chi}\leq0$
is $N\left(T,\overline{\chi}\right)$.
\begin{thm}
\label{thm:dobule series =00005Cchi}Assume the generalized Riemann
hypothesis, fix $q_{1},q_{2}\in\mathbb{N}^{+}$ and $\chi_{1}\mod q_{1},\,\chi_{2}\mod q_{2}$
two Dirichlet character. Let $\rho_{\chi_{j}}=\frac{1}{2}+i\gamma_{\chi_{j}}$
runs over the non-trivial zeros of the $L$- function $L(s,\chi_{j})$,
$j=1,2$. Let 
\[
Z_{0}^{+}\left(\chi_{j}\right):=\left\{ \rho_{\chi_{j}}:L\left(\rho_{\chi_{j}},\chi_{j}\right)=0,\,\gamma_{\chi_{j}}\geq0\right\} ,
\]
\[
Z^{+}\left(\chi_{j}\right):=\left\{ \rho_{\chi_{j}}:L\left(\rho_{\chi_{j}},\chi_{j}\right)=0,\,\gamma_{\chi_{j}}>0\right\} 
\]
and assume that with the notation
\[
\sum_{\rho_{\chi_{j}}\in Z_{0}^{+}\left(\chi_{j}\right)},\,\sum_{\rho_{\chi_{j}}\in Z^{+}\left(\chi_{j}\right)}
\]
we intend that we are taking $\rho_{\chi_{j}}\in Z_{0}^{+}\left(\chi_{j}\right),\,\rho_{\chi_{j}}\in Z^{+}\left(\chi_{j}\right)$
and summing it considering their multiplicity, for $j=1,2$. Let $f\left(z\right),\,g\left(z\right)$
be complex functions defined on the whole set of the non-trivial zeros
of $L\left(s,\chi_{1}\right)$, $L\left(s,\overline{\chi}_{1}\right)$
and $L\left(s,\chi_{2}\right)$, $L\left(s,\overline{\chi}_{2}\right)$,
respectively. Moreover, assume
\[
\sum_{\underset{{\scriptstyle \gamma_{\chi_{1}}\leq T}}{\rho_{\chi_{1}}\in Z_{0}^{+}\left(\chi_{1}\right)}}\left|\frac{f\left(\rho_{\chi_{1}}\right)}{\rho_{\chi_{1}}}\right|=o\left(T^{\alpha}\right),\,\sum_{\underset{{\scriptstyle \gamma_{\chi_{1}}\leq T}}{\rho_{\chi_{1}}\in Z_{0}^{+}\left(\chi_{1}\right)}}\left|\frac{f\left(\overline{\rho_{\chi_{1}}}\right)}{\overline{\rho_{\chi_{1}}}}\right|=o\left(T^{\alpha}\right)
\]
\[
\sum_{\underset{{\scriptstyle \gamma_{\chi_{2}}\leq T}}{\rho_{\chi_{2}}\in Z_{0}^{+}\left(\chi_{2}\right)}}\left|\frac{g\left(\rho_{\chi_{2}}\right)}{\rho_{\chi_{2}}}\right|=o\left(T^{\alpha}\right),\,\sum_{\underset{{\scriptstyle \gamma_{\chi_{2}}\leq T}}{\rho_{\chi_{2}}\in Z_{0}^{+}\left(\chi_{2}\right)}}\left|\frac{g\left(\overline{\rho_{\chi_{2}}}\right)}{\overline{\rho_{\chi_{2}}}}\right|=o\left(T^{\alpha}\right)
\]
as $T\rightarrow+\infty$ for every $\alpha>0$. Then, the double
series
\[
\sum_{\rho_{\chi_{1}}}f\left(\rho_{\chi_{1}}\right)\sum_{\rho_{\chi_{2}}}g\left(\rho_{\chi_{2}}\right)\frac{\Gamma\left(\rho_{\chi_{1}}\right)\Gamma\left(\rho_{\chi_{2}}\right)}{\Gamma\left(\rho_{\chi_{1}}+\rho_{\chi_{2}}+k+1\right)}
\]
converges absolutely for $k\in\mathbb{R},\,k>1/2$.
\end{thm}

\begin{proof}
Fix $\chi_{1}\mod q_{1},\,\chi_{2}\mod q_{2}$. We may assume that
$\gamma_{\chi_{j}}\neq0,\,j=1,2$ because, if one them is equal to
zero, we have
\[
\sum_{\rho_{\chi_{1}}}f\left(\rho_{\chi_{1}}\right)\sum_{\rho_{\chi_{2}}:\gamma_{\chi_{2}}=0}g\left(\rho_{\chi_{2}}\right)\frac{\Gamma\left(\rho_{\chi_{1}}\right)\Gamma\left(\rho_{\chi_{2}}\right)}{\Gamma\left(\rho_{\chi_{1}}+\rho_{\chi_{2}}+k+1\right)}
\]
\[
\ll_{f}N\left(2,\chi_{2}\right)\sum_{\rho_{\chi_{1}}}\left|f\left(\rho_{\chi_{1}}\right)\right|\left|\frac{\Gamma\left(\rho_{\chi_{1}}\right)}{\Gamma\left(\rho_{\chi_{1}}+k+\frac{3}{2}\right)}\right|
\]
and the series, for $k>1/2$, trivially converges (actually it converges
for smaller $k$, but it is not important for our aims) by the Stirling's
formula
\begin{equation}
\left|\Gamma\left(x+iy\right)\right|\sim e^{-\pi\left|y\right|/2}\left|y\right|^{x-1/2},\,x_{1}<x<x_{2},\,\left|y\right|\rightarrow+\infty\label{eq:stirling}
\end{equation}
(see e.g., \cite{TIT}, section 4.42), and the same argument holds
in the case $\sum_{\rho_{\chi_{1}}:\gamma_{\chi_{1}}=0}\sum_{\rho_{\chi_{2}}}$,
hence we can assume $\gamma_{\chi_{j}}\neq0,\,j=1,2$. We start with
the case $\gamma_{\chi_{1}}>0,\,\gamma_{\chi_{2}}>0$. By the Stirling's
formula (\ref{eq:stirling}), we have
\[
\sum_{\rho_{\chi_{1}}\in Z^{+}\left(\chi_{1}\right)}\left|f\left(\rho_{\chi_{1}}\right)\right|\sum_{\rho_{\chi_{2}}\in Z^{+}\left(\chi_{2}\right)}\left|g\left(\rho_{\chi_{2}}\right)\right|\left|\frac{\Gamma\left(\rho_{\chi_{1}}\right)\Gamma\left(\rho_{\chi_{2}}\right)}{\Gamma\left(\rho_{\chi_{1}}+\rho_{\chi_{2}}+k+1\right)}\right|
\]
\[
\ll\sum_{\rho_{\chi_{1}}\in Z^{+}\left(\chi_{1}\right)}\left|f\left(\rho_{\chi_{1}}\right)\right|\sum_{\rho_{\chi_{2}}\in Z^{+}\left(\chi_{2}\right)}\left|g\left(\rho_{\chi_{2}}\right)\right|\frac{1}{\left(\gamma_{\chi_{1}}+\gamma_{\chi_{2}}\right)^{3/2+k}}
\]
\[
\ll\sum_{\rho_{\chi_{1}}\in Z^{+}\left(\chi_{1}\right)}\frac{\left|f\left(\rho_{\chi_{1}}\right)\right|}{\gamma_{\chi_{1}}^{3/4+k/2}}\sum_{\rho_{\chi_{2}}\in Z^{+}\left(\chi_{2}\right)}\frac{\left|g\left(\rho_{\chi_{2}}\right)\right|}{\gamma_{\chi_{2}}^{3/4+k/2}}
\]
by arithmetic-geometric mean, and both series converges absolutely
if $k>1/2$ since, by partial summation, taking $T>0$, we have
\[
\sum_{\rho_{\chi_{1}}:0<\gamma_{\chi_{1}}\leq T}\frac{\left|f\left(\rho_{\chi_{1}}\right)\right|}{\gamma_{\chi_{1}}^{3/4+k/2}}\ll\sum_{\rho_{\chi_{1}}:0<\gamma_{\chi_{1}}\leq T}\left|\frac{f\left(\rho_{\chi_{1}}\right)}{\rho_{\chi_{1}}}\right|\frac{1}{\gamma_{\chi_{1}}^{k/2-1/4}}
\]
\begin{equation}
=\sum_{\rho_{\chi_{1}}:0<\gamma_{\chi_{1}}\leq T}\left|\frac{f\left(\rho_{\chi_{1}}\right)}{\rho_{\chi_{1}}}\right|\frac{1}{T^{k/2-1/4}}+\frac{2k-1}{4}\int_{\delta_{\chi_{1}}}^{T}\sum_{\rho_{\chi_{1}}:0<\gamma_{\chi_{1}}\leq t}\left|\frac{f\left(\rho_{\chi_{1}}\right)}{\rho_{\chi_{1}}}\right|t^{-k/2-3/4}dt\label{eq:PS double series}
\end{equation}
where $\delta_{\chi_{1}}>0$ is a number such that
\[
0<\delta_{\chi_{1}}\leq\min\left\{ \gamma_{\chi_{1}}>0:L\left(\frac{1}{2}+i\gamma_{\chi_{1}},\chi_{1}\right)=0\right\} .
\]
and the same considerations holds for the case $g\left(\rho_{\chi_{2}}\right)$.
Now, we consider the case $\gamma_{\chi_{1}}>0,\,\gamma_{\chi_{2}}<0.$
Clearly, if $\rho_{\chi_{2}}=1/2+i\gamma_{\chi_{2}}$ is a zero of
$L\left(s,\chi_{2}\right)$, then $\overline{\rho_{\chi_{2}}}$ is
a zero of $L\left(s,\overline{\chi_{2}}\right)$. Let us define
\[
\rho_{\chi_{2}}^{*}:=\overline{\rho_{\overline{\chi_{2}}}}.
\]
Then, we have to deal with
\[
\sum_{\rho_{\chi_{1}}\in Z^{+}\left(\chi_{1}\right)}\left|f\left(\rho_{\chi_{1}}\right)\right|\sum_{\rho_{\overline{\chi_{2}}}\in Z^{+}\left(\overline{\chi_{2}}\right)}\left|g\left(\rho_{\chi_{2}}^{*}\right)\right|\left|\frac{\Gamma\left(\rho_{\chi_{1}}\right)\Gamma\left(\rho_{\chi_{2}}^{*}\right)}{\Gamma\left(\rho_{\chi_{1}}+\rho_{\chi_{2}}^{*}+k+1\right)}\right|.
\]
Let $0<\alpha<1$ be fixed. We split the second series in the following
way:
\[
\sum_{\rho_{\chi_{2}}\in Z^{+}\left(\overline{\chi_{2}}\right)}=\sum_{\underset{{\scriptstyle \left|\gamma_{\chi_{1}}-\gamma_{\chi_{2}}\right|<\alpha\max\left(\gamma_{\chi_{1}},\gamma_{\chi_{2}}\right)}}{\rho_{\overline{\chi_{2}}}\in Z^{+}\left(\overline{\chi_{2}}\right)}}+\sum_{\underset{{\scriptstyle \left|\gamma_{\chi_{1}}-\gamma_{\chi_{2}}\right|\geq\alpha\max\left(\gamma_{\chi_{1}},\gamma_{\chi_{2}}\right)}}{\rho_{\overline{\chi_{2}}}\in Z^{+}\left(\overline{\chi_{2}}\right)}}
\]
\[
=:\sideset{}{_{1}}\sum+\sideset{}{_{2}}\sum,
\]
say. We start with $\sideset{}{_{1}}\sum$. From the relation
\[
\left|\Gamma\left(x+iy\right)\right|\geq\Gamma\left(x\right)\text{sech}\left(\pi y\right)^{1/2},\,x\geq1/2
\]

(see \cite{O2010}, relation $5.6.7$) we have
\[
\left|\Gamma\left(\rho_{\chi_{1}}+\rho_{\chi_{2}}^{*}+k+1\right)\right|\geq\Gamma\left(2+k\right)\min_{0\leq t\leq\alpha\max\left(\gamma_{\chi_{1}},\gamma_{\chi_{2}}\right)}\text{sech}\left(\pi t\right)^{1/2}
\]
\[
\gg_{k}e^{-\alpha\pi\left(\gamma_{\chi_{1}}+\gamma_{\chi_{2}}\right)/2}
\]
hence, using again (\ref{eq:stirling}), we get
\[
\sum_{\rho_{\chi_{1}}\in Z^{+}\left(\chi_{1}\right)}\left|f\left(\rho_{\chi_{1}}\right)\right|\sum_{\underset{{\scriptstyle \left|\gamma_{\chi_{1}}-\gamma_{\chi_{2}}\right|<\alpha\max\left(\gamma_{\chi_{1}},\gamma_{\chi_{2}}\right)}}{\rho_{\overline{\chi_{2}}}\in Z^{+}\left(\overline{\chi_{2}}\right)}}\left|g\left(\rho_{\chi_{2}}^{*}\right)\right|\left|\frac{\Gamma\left(\rho_{\chi_{1}}\right)\Gamma\left(\rho_{\chi_{2}}^{*}\right)}{\Gamma\left(\rho_{\chi_{1}}+\rho_{\chi_{2}}^{*}+k+1\right)}\right|
\]
\[
\ll_{k}\sum_{\rho_{\chi_{1}}\in Z^{+}\left(\chi_{1}\right)}\left|f\left(\rho_{\chi_{1}}\right)\right|e^{-\pi\left(1-\alpha\right)\gamma_{\chi_{1}}/2}\sum_{\rho_{\overline{\chi_{2}}}\in Z^{+}\left(\overline{\chi_{2}}\right)}\left|g\left(\rho_{\chi_{2}}^{*}\right)\right|e^{-\pi\left(1-\alpha\right)\gamma_{\chi_{2}}/2}
\]
and the convergence follows again by partial summation. Now we consider
$\sideset{}{_{2}}\sum$. We split in two further cases: if $\left|\gamma_{\chi_{1}}-\gamma_{\chi_{2}}\right|\geq\alpha\max\left(\gamma_{\chi_{1}},\gamma_{\chi_{2}}\right)$,
then $\gamma_{\chi_{1}}\neq\gamma_{\chi_{2}}$, so if $\gamma_{\chi_{1}}>\gamma_{\chi_{2}}$,
then again by (\ref{eq:stirling})
\[
\sum_{\rho_{\chi_{1}}\in Z^{+}\left(\chi_{1}\right)}\left|f\left(\rho_{\chi_{1}}\right)\right|\sum_{\underset{{\scriptstyle \left|\gamma_{\chi_{1}}-\gamma_{\chi_{2}}\right|\geq\alpha\gamma_{\chi_{1}}}}{\rho_{\overline{\chi_{2}}}\in Z^{+}\left(\overline{\chi_{2}}\right)}}\left|g\left(\rho_{\chi_{2}}^{*}\right)\right|\left|\frac{\Gamma\left(\rho_{\chi_{1}}\right)\Gamma\left(\rho_{\chi_{2}}^{*}\right)}{\Gamma\left(\rho_{\chi_{1}}+\rho_{\chi_{2}}^{*}+k+1\right)}\right|
\]
\[
\ll\sum_{\rho_{\chi_{1}}\in Z^{+}\left(\chi_{1}\right)}\left|f\left(\rho_{\chi_{1}}\right)\right|\sum_{\underset{{\scriptstyle \left|\gamma_{\chi_{1}}-\gamma_{\chi_{2}}\right|\geq\alpha\gamma_{\chi_{1}}}}{\rho_{\overline{\chi_{2}}}\in Z^{+}\left(\overline{\chi_{2}}\right)}}\left|g\left(\rho_{\chi_{2}}^{*}\right)\right|\frac{e^{-\pi\gamma_{\chi_{2}}/2}e^{-\pi\gamma_{\chi_{1}}/2}}{\left(\gamma_{\chi_{1}}-\gamma_{\chi_{2}}\right)^{3/2+k}e^{-\pi\left(\gamma_{\chi_{1}}-\gamma_{\chi_{2}}\right)/2}}
\]
\[
\ll_{\alpha}\sum_{\rho_{\chi_{1}}\in Z^{+}\left(\chi_{1}\right)}\frac{\left|f\left(\rho_{\chi_{1}}\right)\right|}{\gamma_{\chi_{1}}^{3/2+k}}\sum_{\rho_{\overline{\chi_{2}}}\in Z^{+}\left(\overline{\chi_{2}}\right)}\left|g\left(\rho_{\chi_{2}}^{*}\right)\right|e^{-\pi\gamma_{\chi_{2}}}
\]
and again the convergence follows. A similar calculation holds in
the case $\gamma_{\chi_{2}}>\gamma_{\chi_{1}}$. 
\end{proof}
We also recall the following very easy fact, because we will use it
very often.
\begin{prop}
\label{prop:trivial}For every $a>0,\,\alpha\geq0$ and $n\geq0$
we have
\[
\int_{0}^{a}x^{\alpha}\left|\log\left(x\right)\right|^{n}dx\ll_{n}a^{\alpha+1}\left(1+\left|\log\left(a\right)\right|^{n}\right).
\]
\end{prop}

\subsection{Explicit formulae extended up to $x>0$}

In this part we derive explicit formulae for some classical arithmetic
functions. For our purposes, it is important to formulate these identities
for all $x>0$, rather than only in the classical ranges $x>1$ or
$x>2$. At first sight this may seem unnecessary, since for small
values of $x$, the corresponding summatory functions are often trivial
(for example, $\psi(x):=\sum_{n\leq x}\Lambda(n)=0$ if $0<x<2$).

However, what is really needed is not so much an extension of the
summatory functions themselves, but rather a formulation of the zero-sum
side that remains meaningful when $x$ is small, without having to
rely on estimates for non-absolutely convergent series. For example,
one has an explicit formula of the shape
\begin{equation}
\psi\left(x\right)=x-\sum_{\rho:\left|\gamma\right|\leq T}\frac{x^{\rho}}{\rho}+R\left(x,T\right)\label{eq:trunc psi}
\end{equation}
valid for $x,T\geq2$, for a suitable error $R\left(x,T\right)$.
Even if the main term $x-\sum_{\rho:\left|\gamma\right|\leq T}\frac{x^{\rho}}{\rho}$
can formally be written also for $0<x<2$ such an extension is not
harmless, since the sum over the zeros is not absolutely convergent
and may produce an error term depending on $T$ that worsens when
$T$ grows. Moreover, in our setting because, by Theorem \ref{thm:MAINCGZ}
we are led to consider convolutions of the form

\[
\int x^{\rho_{1}}\left(A-x\right)^{\rho_{2}}dx
\]
and if the range of integration does not start at $0$, then incomplete
beta functions naturally appear, making the resulting expressions
substantially less transparent than the corresponding full beta-type
convolution.

We start by recalling the explicit formula for $\psi\left(x,\chi\right):=\sum_{n\leq x}\Lambda(n)\chi(n)$
and extending it to $x>0$.
\begin{thm}
\label{thm:Explicit psi}(Truncated explicit formula for $\psi\left(x,\chi\right)$)
Assume GRH. Let $x>0,\,x\neq1$,and $T\geq2$. Let $\chi$ be a primitive
non-principal character mod $q$, $q\in\mathbb{N},$ $q>1$. Then
\[
\psi\left(x,\chi\right)=-\sum_{\rho_{\chi}:\left|\gamma_{\chi}\right|\leq T}\frac{x^{\rho_{\chi}}}{\rho_{\chi}}+R_{\chi}\left(x,T,q\right)
\]
where $\rho_{\chi}$ runs over the non-trivial zeros of the function
$L(s,\chi)$, with 
\[
R_{\chi}\left(x,T,q\right)\ll\log(q)+\left|\log\left(x\right)\right|+\left|\log\left(\left|x-1\right|\right)\right|+\frac{\left(x+1\right)\log\left(qT\max\left\{ x,1/x\right\} \right)^{2}}{T}
\]
and the implicit constant is absolute.
\end{thm}

\begin{proof}
If $x>1$ the claim follows from the classical truncated explicit
formula. Indeed, defining
\[
\psi_{0}\left(x,\chi\right):=\psi\left(x,\chi\right)-\frac{1}{2}\Lambda\left(x\right)\chi\left(x\right)
\]
 we have
\[
\psi_{0}\left(x,\chi\right)=-\sum_{\rho_{\chi}:\left|\gamma_{\chi}\right|\leq T}\frac{x^{\rho_{\chi}}}{\rho_{\chi}}-\frac{1}{2}\log\left(x-1\right)-\frac{\chi\left(-1\right)}{2}\log\left(x+1\right)+C\left(\chi\right)+R_{\chi}^{1}\left(x,T,q\right)
\]
where
\[
R_{\chi}^{1}\left(x,T,q\right)\ll\log\left(x\right)\min\left(1,\frac{x}{T\left\langle x\right\rangle }\right)+\frac{x}{T}\log\left(qTx\right)^{2}
\]
and 
\[
C\left(\chi\right):=\frac{L^{\prime}}{L}\left(1,\overline{\chi}\right)+\log\left(\frac{q}{2\pi}\right)-C_{0}
\]
where $C_{0}$ is the Euler-Mascheroni constant (see \cite{MV2006},
Corollary 12.11) and $\left\langle x\right\rangle $ means the distance
of $x$ to the nearest integer. The estimate $C\left(\chi\right)\ll\log(q)$
holds if there is no exceptional zero (see again Corollary 12.11).
For the case $0<x<1$, we consider $y>1$ and $\sigma_{0}=1/\log\left(y\right)$
we take
\[
\sideset{}{^{\prime}}\sum_{n\leq y}\frac{\overline{\chi}\left(n\right)\Lambda\left(n\right)}{n}=-\frac{1}{2\pi i}\int_{\sigma_{0}-iT}^{\sigma_{0}+iT}\frac{L^{\prime}}{L}\left(s+1,\overline{\chi}\right)\frac{y^{s}}{s}ds
\]
\[
+O\left(\frac{\log\left(y\right)}{y}\min\left(1,\frac{y}{T\left\langle y\right\rangle }\right)+\frac{\log\left(y\right)^{2}}{T}\right)
\]
by Theorem 5.2 and Corollary 5.3 of \cite{MV2006}. From now on, we
denote with $\sigma:=\text{Re}(s)$, for $s\in\mathbb{C}.$ Let $\kappa$
be equal to $0$ or $1$ if $\overline{\chi}\left(-1\right)=1$ or
$\overline{\chi}\left(-1\right)=-1$, respectively, and let $\mathcal{A}\left(\kappa\right)$
be the set of points $s\in\mathbb{C}$ such that $\sigma\leq-1$ and
$\left|s+2n-\kappa\right|\geq1/4$. Let $\mathscr{C}$ be the contour
consisting of the line segments connecting $\sigma_{0}-iT_{1}$, $-K-iT_{1}$,
$-K+iT_{1}$ and $\sigma_{0}+iT_{1}$ where $K$ is chosen so that
$K-\kappa$ is an odd positive integer and $T_{1}$ is chosen as in
Lemma $12.9$ of \cite{MV2006}, in particular $T_{1}\in[T,T+1]$.
Now, for the horizontal segment, by Lemma $12.7$ and $12.9$ of \cite{MV2006},
we have
\[
\ll\int_{-1+iT_{1}}^{\sigma_{0}+iT_{1}}\left|\frac{L^{\prime}}{L}\left(s+1,\overline{\chi}\right)\frac{y^{s}}{s}\right|\left|ds\right|+\int_{-K+iT_{1}}^{-1+iT_{1}}\left|\frac{L^{\prime}}{L}\left(s+1,\overline{\chi}\right)\frac{y^{s}}{s}\right|\left|ds\right|
\]
\[
\ll\frac{\log\left(qT\right)^{2}}{T}\int_{-2}^{\sigma_{0}}y^{\sigma}d\sigma+\frac{\log\left(2qT\right)}{T}\int_{-\infty}^{-1}y^{\sigma}d\sigma
\]
\[
\ll\frac{\log\left(qT\right)^{2}}{T\log\left(y\right)}+\frac{\log\left(2qT\right)}{T\log\left(y\right)}
\]
and since, from the vertical segment
\[
\int_{-K-iT_{1}}^{-K+iT_{1}}\left|\frac{L^{\prime}}{L}\left(s+1,\overline{\chi}\right)\frac{y^{s}}{s}\right|\left|ds\right|\ll y^{-K}\frac{\log\left(qK\right)T}{K}\rightarrow0,\,K\rightarrow+\infty
\]
we get, by the residue theorem, that
\[
\sideset{}{^{\prime}}\sum_{n\leq y}\frac{\overline{\chi}\left(n\right)\Lambda\left(n\right)}{n}=-\sum_{\rho_{\overline{\chi}}:\left|\gamma_{\overline{\chi}}\right|<T}\frac{y^{\rho_{\overline{\chi}}-1}}{1-\rho_{\overline{\chi}}}+\sum_{k\geq1}\frac{y^{-2k-1+\kappa}}{2k+1-\kappa}-\frac{L^{\prime}}{L}\left(1,\overline{\chi}\right)+\frac{\overline{\chi}\left(-1\right)+1}{2y}+R_{\overline{\chi}}^{2}\left(y,T,q\right)
\]
with
\[
R_{\overline{\chi}}^{2}\left(y,T,q\right)\ll\frac{\log\left(y\right)}{y}\min\left(1,\frac{y}{T\left\langle y\right\rangle }\right)+\frac{\log\left(qyT\right)^{2}}{T}.
\]
Now, using the identities 
\[
\sum_{k\geq1}\frac{y^{-2k}}{2k}=-\frac{1}{2}\log\left(1-y^{-2}\right),\,\sum_{k\geq1}\frac{y^{-2k-1}}{2k+1}=\frac{1}{2}\log\left(\frac{y+1}{y-1}\right)-\frac{1}{y}
\]
we are able to control the sum $\sum_{k\geq1}\frac{y^{-2k-1+\kappa}}{2k+1-\kappa}$.
Then, taking $y=1/x$ and using the symmetries of the non-trivial
zeros, that is, under GRH, if $\rho_{\overline{\chi}}$ is a non-trivial
zero of $L\left(s,\overline{\chi}\right)$, then $1-\rho_{\overline{\chi}}$
is a zero of $L\left(s,\chi\right)$, we have that
\[
-\sum_{\rho_{\chi}:\left|\gamma_{\chi}\right|<T}\frac{x^{\rho_{\chi}}}{\rho_{\chi}}=O\left(\sideset{}{^{\prime}}\sum_{n\leq1/x}\frac{\Lambda\left(n\right)}{n}+\left|\frac{L^{\prime}}{L}\left(1,\overline{\chi}\right)\right|+\left|\log\left(\frac{x-1}{x}\right)\right|\right.
\]
\[
\left.+\left|\log\left(x\right)\right|\left[1+\min\left(1,\frac{x}{T\left\langle x\right\rangle }\right)+\min\left(1,\frac{1}{xT\left\langle 1/x\right\rangle }\right)\right]+\frac{\log\left(qT/x\right)^{2}}{T}\right)
\]
and, since $\left|\frac{L^{\prime}}{L}\left(1,\chi\right)\right|=\left|\frac{L^{\prime}}{L}\left(1,\overline{\chi}\right)\right|\ll\log(q)$
by Corollary 12.11 and observing that $\sideset{}{^{\prime}}\sum_{n\leq1/x}\frac{\Lambda\left(n\right)}{n}\ll\left|\log\left(x\right)\right|$
and $\psi_{0}\left(x,\chi\right)=0$ if $0<x<1$, the thesis follows.
\end{proof}
Due to the presence of the M\"obius function and our use of explicit
formulae, it is natural to introduce an additional conjectural assumption
concerning the simplicity of the zeros of the Riemann zeta function.
To address this issue, we recall the following conjecture.
\begin{conjecture}
(Gonek-Hejhal conjecture \cite{G1989,H1987}) Let $\rho=\beta+i\gamma$
be a non-trivial zero of the Riemann zeta function $\zeta(s)$. Let
\[
J_{k}\left(T\right):=\sum_{0<\gamma\leq T}\frac{1}{\left|\zeta^{\prime}\left(\rho\right)\right|^{2k}}.
\]
Then, for all $k\in\mathbb{R}$, we have
\begin{equation}
J_{k}\left(T\right)\asymp T\log\left(T\right)^{(k-1)^{2}}.\label{eq:J_k upp low bound}
\end{equation}
\end{conjecture}

We recall that this conjecture implies that all non-trivial zeros
of the Riemann zeta function are simple (see \cite{H2000}). Note
that some refinement by Hughes, Keating and O\textquoteright Connell
suggests that actually the conjecture is false if $k\geq3/2$ (see
again \cite{HKO'C}). Actually, we just need the special case $k=1$;
for this reason, we use the following simpler conjecture.
\begin{conjecture}
\label{conj:simple zeros} We have that
\[
J_{1}\left(T\right)=\sum_{0<\gamma\leq T}\frac{1}{\left|\zeta^{\prime}\left(\rho\right)\right|^{2}}\ll T.
\]
\end{conjecture}

Note that also this conjecture implies that all the zeros of the $\zeta(s)$
are simple (see \cite{H2000} Conjecture 3.9). Moreover, we have,
by the Cauchy-Schwarz inequality and the bound (\ref{eq:N(T)})

\begin{equation}
\sum_{0<\gamma\leq T}\frac{1}{\left|\zeta^{\prime}\left(\rho\right)\right|}\ll T\log\left(T\right)^{1/2}\label{eq:1/zeta^=00005Cprime}
\end{equation}
which is off from (\ref{eq:J_k upp low bound}) by an exponent $1/4$
on the $\log\left(T\right)$. It is not a real problem, for our aims.
Hence, we can work directly with Conjecture \ref{conj:simple zeros}. 

Now we are able to recall the explicit formula for the Mertens' function
$M(x):=\sum_{n\leq x}\mu\left(n\right)$. Also in this case, we need
to extend the formula to all $x>0$, excluding at most the case $x=1$.
\begin{thm}
\label{thm:Explicit M}(Truncated explicit formula for $M(x)$)

Let $0<\varepsilon<1/2$ be fixed. Assume RH, $x>0,\,x\neq1$, $U\geq2$
and Conjecture \ref{conj:simple zeros}. Then, we have
\begin{equation}
M(x)=\sum_{\rho:\left|\gamma\right|\leq U}\frac{x^{\rho}}{\zeta^{\prime}(\rho)\rho}+O_{\varepsilon}\left(1+x^{\varepsilon}+\frac{x\left(\left|\log\left(x\right)\right|+1\right)}{U}+\frac{x-x^{\varepsilon}}{U^{1-\varepsilon}\log\left(x\right)}\right).\label{eq:truncated M}
\end{equation}
\end{thm}

This formula can be found in \cite{CGZ2026} and, for $x\geq2$, in
\cite{Ng}, with some small differences in the error term, which however
can be absorbed in (\ref{eq:truncated M}). Note that, for our purposes,
it is important to reduce the number of terms in the error term, in
order to simplify the subsequent calculations. Indeed, we may actually
write
\begin{equation}
M(x)=\sum_{\rho:\left|\gamma\right|\leq U}\frac{x^{\rho}}{\zeta^{\prime}(\rho)\rho}+O\left(1+x^{\varepsilon}+\frac{\left(x^{\varepsilon}+x\right)\left(\left|\log\left(x\right)\right|+1\right)}{U^{1-\varepsilon}}\right)\label{eq:Explitic M eq}
\end{equation}
where we use the elementary inequality
\[
\frac{x-x^{\varepsilon}}{\log\left(x\right)}\leq x+x^{\varepsilon},\,x>0,\,x\neq1.
\]

Since we are interested to find an upper bound for weighted average
of diagonal twisted EH problem also in the log's case, we also need
an explicit formula for
\[
\widetilde{M}\left(x\right):=\sum_{n\leq x}\mu\left(n\right)\log\left(n\right)
\]
and, luckily, this can be derived by Theorem \ref{thm:Explicit M}.
Indeed, if $x>1$, by partial summation,
\begin{equation}
\widetilde{M}\left(x\right)=\log\left(x\right)M\left(x\right)-\int_{1}^{x}\frac{M\left(t\right)}{t}dt\label{eq:M tilde}
\end{equation}
and inserting the explicit formula (\ref{eq:Explitic M eq}) we get
\[
\widetilde{M}\left(x\right)=\sum_{\rho:\left|\gamma\right|\leq U}\frac{x^{\rho}}{\zeta^{\prime}(\rho)\rho}\left(\log\left(x\right)-\frac{1}{\rho}\right)+O_{\varepsilon}\left(\log\left(x\right)+x^{\varepsilon}+\frac{\left(x^{\varepsilon}+x\right)\left(\log\left(x\right)\right)}{U^{1-\varepsilon}}\right).
\]
Now, if $0<x<1$, we trivially observe that
\begin{equation}
0=\widetilde{M}\left(x\right)=\log\left(x\right)M\left(x\right)=\log\left(x\right)M\left(x\right)-\sum_{\rho:\left|\gamma\right|\leq U}\frac{x^{\rho}}{\zeta^{\prime}(\rho)\rho^{2}}\label{eq:ext trick}
\end{equation}
\[
+O_{\varepsilon}\left(1+\left|\log\left(x\right)\right|+x^{\varepsilon}+\frac{\left(x^{\varepsilon}+x\right)\left(\left|\log\left(x\right)\right|+1\right)}{U^{1-\varepsilon}}\right)
\]
since $\sum_{\rho}\frac{1}{\left|\zeta^{\prime}(\rho)\rho^{2}\right|}$
is absolutely convergent, by Conjecture \ref{conj:simple zeros},
and so its contribute is absorbed in the $O_{\varepsilon}(1)$ term.
Then, inserting the explicit formula in (\ref{eq:ext trick}) and
recalling that, if $0<x<1$, $x^{A}\left|\log\left(x\right)\right|^{B}\ll1$
for every fixed $A,B>0$, we get
\begin{equation}
\widetilde{M}\left(x\right)=\sum_{\rho:\left|\gamma\right|\leq U}\frac{x^{\rho}}{\zeta^{\prime}(\rho)\rho}\left(\log\left(x\right)-\frac{1}{\rho}\right)+O_{\varepsilon}\left(1+\left|\log\left(x\right)\right|+x^{\varepsilon}+\frac{\left(x^{\varepsilon}+x\right)\left(\left|\log\left(x\right)\right|+1\right)}{U^{1-\varepsilon}}\right)\label{eq:M tilde explicit}
\end{equation}
valid for $x>0$, $x\neq1$.

\section{weighted averages of diagonal twisted eh problem with $g(n)=\mu(n)$}

Now we are able to study general weighted sum
\[
\sum_{n}\sum_{m}\Lambda\left(n\right)\chi\left(n\right)\mu\left(m\right)f\left(\frac{n+m}{N}\right)
\]
where $N\geq4$ is a natural number. We divide this section in two
parts: in the first one, we will use weights in the Sobolev space
$W^{2,1}$, and this regularity allows to get a bound for the diagonal
twisted EH for all $0<\theta<1$; in the second part, we will use
weights that belong to the H\"older-Zygmud spaces $\mathcal{C}^{\delta}$,
$\delta\in[1,2]$, and we pay the minor regularity with a smaller
range of $\theta$. 

\subsection{Weight in Sobolev space $W^{2,1}$}

Let $f:\text{\ensuremath{\mathbb{R}}}\rightarrow\mathbb{\mathbb{C}}$
be a function that verifies the hypotheses of Theorem \ref{thm:MAINCGZ}.
Since we are interested in weights that have support in $\mathbb{R}_{0}^{+}$,
From now on we will assume that $\alpha=0$ and $\beta\in\mathbb{R}^{+}$.
This choice also has the advantage of being able to use the Theorem
\ref{thm:MAINCGZ} and Theorem \ref{thm:MainDiscrete} without boundary
terms and without creating dependencies between $\alpha,\beta$ and
$N$. The price to pay will be having to assume that the second derivative
of $f$ is bounded in a neighborhood of the origin, due to the logarithmic
terms that will emerge from the various estimates. 

We recall that hypotheses of Theorem \ref{thm:MAINCGZ} imply, essentially,
that $f_{\vert\left(0,\beta\right)}\in W^{2,1}\left(0,\beta\right)$.
We have, taking $g_{1}(n)=\Lambda(n)\chi(n),\,g_{2}(m)=\mu(m)$, that
\[
\sum_{n\leq N\beta}\sum_{m\leq N\beta-n}\Lambda\left(n\right)\chi\left(n\right)\mu\left(m\right)f\left(\frac{n+m}{N}\right)=\frac{1}{N}\int_{0}^{\beta}f^{\prime\prime}\left(w\right)\int_{0}^{Nw}M\left(s\right)\psi\left(Nw-s,\chi\right)dsdw
\]
so we have
\[
\frac{1}{\varphi\left(q\right)}\sum_{\chi\neq\chi_{0}}\overline{\chi}\left(N\right)\sum_{n\leq N\beta}\sum_{m\leq N\beta-n}\Lambda\left(n\right)\chi\left(n\right)\mu\left(m\right)f\left(\frac{n+m}{N}\right)
\]
\[
=\frac{1}{N\varphi\left(q\right)}\int_{0}^{\beta}f^{\prime\prime}\left(w\right)\int_{0}^{Nw}M\left(s\right)\sum_{\chi\neq\chi_{0}}\overline{\chi}\left(N\right)\psi\left(Nw-s,\chi\right)dsdw.
\]
 We recall the abuse of notation: the sum $\sum_{\chi\neq\chi_{0}}$
is intended on all characters $\mod q$ other than the principal one
$\chi_{0}$ but $\chi\left(n\right)$ is the primitive character that
induces $\chi$. So, if we apply the previous relation to the weighted
average of the diagonal twisted EH problem, we get
\[
\sum_{\underset{{\scriptstyle \left(N,q\right)=1}}{1<q\leq N^{\theta}}}\frac{1}{\varphi\left(q\right)}\left|\sum_{\chi\neq\chi_{0}}\overline{\chi}\left(N\right)\sum_{n\leq N\beta}\sum_{m\leq N\beta-n}\Lambda\left(n\right)\chi\left(n\right)\mu\left(m\right)f\left(\frac{n+m}{N}\right)\right|
\]
\begin{equation}
=\sum_{\underset{{\scriptstyle \left(N,q\right)=1}}{1<q\leq N^{\theta}}}\frac{1}{\varphi\left(q\right)}\left|\frac{1}{N}\int_{0}^{\beta}f^{\prime\prime}\left(w\right)\int_{0}^{Nw}M\left(s\right)\sum_{\chi\neq\chi_{0}}\overline{\chi}\left(N\right)\psi\left(Nw-s,\chi\right)dsdw\right|.\label{eq:start}
\end{equation}
For studying this object, we firstly prove an explicit formula for
the weighted average.

Before going into details, let's point out that in the following proof,
given the complexity of the calculations, we will prioritize clarity
over the actual logarithmic contribution of the error term: this is
because, for our purposes, it will not significantly influence the
final result.
\begin{thm}
\label{thm:Main1}Fix $0<\varepsilon<1/2$ and assume GRH and Conjecture
\ref{conj:simple zeros}. Let $N\geq4$, $q\in\mathbb{N},\,1<q<N^{\theta},0<\theta<1$.
Moreover, let $\beta>0$ and $f:\text{\ensuremath{\mathbb{R}}}\rightarrow\mathbb{\mathbb{C}}$
satisfying the hypotheses of Theorem \ref{thm:MAINCGZ} with $f^{\prime\prime}$bounded
in a neighborhood of the origin. Then
\[
\frac{1}{\varphi\left(q\right)}\sum_{\chi\neq\chi_{0}}\overline{\chi}\left(N\right)\sum_{n\leq N\beta}\sum_{m\leq N\beta-n}\Lambda\left(n\right)\chi\left(n\right)\mu\left(m\right)f\left(\frac{n+m}{N}\right)
\]
\[
=-\frac{1}{\varphi\left(q\right)}\sum_{\rho:\left|\gamma\right|\leq U}\frac{N^{\rho}\Gamma\left(\rho\right)}{\zeta^{\prime}(\rho)}\sum_{\chi\neq\chi_{0}}\overline{\chi}\left(N\right)\sum_{\rho_{\chi}:\left|\gamma_{\chi}\right|\leq T}\frac{N^{\rho_{\chi}}\Gamma\left(\rho_{\chi}\right)}{\Gamma\left(\rho_{\chi}+\rho+2\right)}\int_{0}^{\beta}f^{\prime\prime}\left(w\right)w^{\rho+\rho_{\chi}+1}dw
\]
\[
+O_{\varepsilon}\left(\frac{1}{N}\int_{0}^{\beta}\left|f^{\prime\prime}\left(w\right)\right|\left[\mathcal{E}_{1}\left(N,w,q,U,T\right)+\mathcal{E}_{2}\left(N,w,q,U,T\right)\right]dw\right)
\]
where
\[
\mathcal{E}_{1}\left(N,w,q,U,T\right):=\varphi\left(q\right)N^{3/2+\varepsilon}\log\left(q\right)^{3}\log\left(N\right)^{3}\left(1+w^{3/2+\varepsilon}\right)\left(1+\left|\log\left(w\right)\right|^{3}\right)
\]
\[
\times\left[1+\frac{N\left(1+w\right)\log\left(qT\right)^{2}}{T}+\frac{N^{1-\varepsilon}\left(w^{1-\varepsilon}+1\right)}{U^{1-\varepsilon}}\right],
\]
\[
\mathcal{E}_{2}\left(N,w,q,U,T\right)=\varphi(q)N^{1+\varepsilon}\log(q)\log\left(N\right)^{3}\left(1+\left|\log\left(w\right)\right|^{3}\right)
\]
\[
\times\left(1+w^{1+\varepsilon}\right)\left[1+\frac{N\log\left(qT\right)^{2}}{T}\right]\left[1+\frac{N^{1-1\varepsilon}}{U^{1-\varepsilon}}\right].
\]
and the implicit constant depends only on $\varepsilon$.
\end{thm}

\begin{proof}
By (\ref{eq:start}), the problem boils down to considering
\[
\frac{1}{N}\int_{0}^{\beta}f^{\prime\prime}\left(w\right)\int_{0}^{Nw}M\left(s\right)\sum_{\chi\neq\chi_{0}}\overline{\chi}\left(N\right)\psi\left(Nw-s,\chi\right)dsdw.
\]
From the explicit formulas we write $\psi\left(x,\chi\right)$ and
$M\left(x\right)$ as
\[
\psi\left(x,\chi\right)=-\sum_{\rho_{\chi}:\left|\gamma_{\chi}\right|\leq T}\frac{x^{\rho_{\chi}}}{\rho_{\chi}}+R_{1,\chi}\left(x,T,q\right),\,M\left(x\right)=\sum_{\rho:\left|\gamma\right|\leq U}\frac{x^{\rho}}{\zeta^{\prime}(\rho)\rho}+R_{2}(x,U)
\]
then, putting 
\[
\mathcal{M}_{1,\chi}(x,T):=-\sum_{\rho_{\chi}:\left|\gamma_{\chi}\right|\leq T}\frac{x^{\rho_{\chi}}}{\rho_{\chi}},\,\mathcal{M}_{2}(x,U):=\sum_{\rho:\left|\gamma\right|\leq U}\frac{x^{\rho}}{\zeta^{\prime}(\rho)\rho}
\]
and observing
\[
\mathcal{M}_{1,\chi}(x,T)=\psi\left(x,\chi\right)-R_{1,\chi}\left(x,T,q\right),\,\mathcal{M}_{2}(x,U)=M\left(x\right)-R_{2}(x,U)
\]
then we have the decomposition
\begin{equation}
\int_{0}^{Nw}M\left(s\right)\sum_{\chi\neq\chi_{0}}\overline{\chi}\left(N\right)\psi\left(Nw-s,\chi\right)ds=\int_{0}^{Nw}\mathcal{M}_{2}(s,U)\sum_{\chi\neq\chi_{0}}\overline{\chi}\left(N\right)\mathcal{M}_{1,\chi}(Nw-s,T)ds\label{eq:dissect1}
\end{equation}
\begin{equation}
+\int_{0}^{Nw}M\left(s\right)\sum_{\chi\neq\chi_{0}}\overline{\chi}\left(N\right)R_{1,\chi}(Nw-s,T,q)ds+\int_{0}^{Nw}R_{2}(s,U)\sum_{\chi\neq\chi_{0}}\overline{\chi}\left(N\right)\psi\left(Nw-s,\chi\right)ds\label{eq:dissect2}
\end{equation}
\begin{equation}
-\int_{0}^{Nw}R_{2}(s,U)\sum_{\chi\neq\chi_{0}}\overline{\chi}\left(N\right)R_{1,\chi}\left(Nw-s,T,q\right)ds.\label{eq:dissect3}
\end{equation}
It is clear that, from the main terms, we get
\[
\frac{1}{N}\int_{0}^{\beta}f^{\prime\prime}\left(w\right)\int_{0}^{Nw}\mathcal{M}_{2}(s,U)\mathcal{M}_{1,\chi}(Nw-s,T)dsdw
\]
\[
=-\frac{1}{N}\sum_{\rho:\left|\gamma\right|\leq U}\frac{1}{\zeta^{\prime}(\rho)\rho}\sum_{\chi\neq\chi_{0}}\overline{\chi}\left(N\right)\sum_{\rho_{\chi}:\left|\gamma_{\chi}\right|\leq T}\frac{1}{\rho_{\chi}}\int_{0}^{\beta}f^{\prime\prime}\left(w\right)\int_{0}^{Nw}s^{\rho}\left(Nw-s\right)^{\rho_{\chi}}dsdw
\]
\[
=-\sum_{\rho:\left|\gamma\right|\leq U}\frac{N^{\rho}\Gamma\left(\rho\right)}{\zeta^{\prime}(\rho)}\sum_{\chi\neq\chi_{0}}\overline{\chi}\left(N\right)\sum_{\rho_{\chi}:\left|\gamma_{\chi}\right|\leq T}\frac{N^{\rho_{\chi}}\Gamma\left(\rho_{\chi}\right)}{\Gamma\left(\rho_{\chi}+\rho+2\right)}\int_{0}^{\beta}f^{\prime\prime}\left(w\right)w^{\rho+\rho_{\chi}+1}dw
\]
so we now focus on the error terms.

Under GRH, we know that
\[
M\left(x\right)\ll_{\varepsilon}x^{1/2+\varepsilon},\,\psi\left(x,\chi\right)\ll x^{1/2}\log\left(x\right)\log\left(qx\right)
\]
for $x\geq2$ and $q\geq1$ (see \cite{TITzeta} p. 371 and Theorem
13.7 of \cite{MV2006}), hence if $x\leq CN$ for some fixed $C>0$,
we have the uniform bounds, actually valid for every $0<x\leq CN$
\begin{equation}
M\left(x\right)\ll N^{1/2+\varepsilon},\,\psi\left(x,\chi\right)\ll N^{1/2}\log\left(N\right)\log\left(qN\right).\label{eq:uniform bounds}
\end{equation}
Let us define
\[
I_{1}=I_{1}(w,N,q,T):=\int_{0}^{Nw}M\left(s\right)\sum_{\chi\neq\chi_{0}}\overline{\chi}\left(N\right)R_{1,\chi}\left(Nw-s,T,q\right)ds
\]
\[
I_{2}=I_{2}(w,N,q,U):=\int_{0}^{Nw}R_{2}(s,U)\sum_{\chi\neq\chi_{0}}\overline{\chi}\left(N\right)\psi\left(Nw-s,\chi\right)ds
\]
\[
I_{3}=I_{3}(w,N,q,T,U):=\int_{0}^{Nw}R_{2}(s,U)\sum_{\chi\neq\chi_{0}}\overline{\chi}\left(N\right)R_{1,\chi}\left(Nw-s,T,q\right)ds.
\]

\textbf{Bounding $I_{1}.$} 

By Theorem \ref{thm:Explicit psi}, since $\log(q)\geq\log(2)\gg1$
and since $N\geq4$, we have
\[
I_{1}\ll_{\varepsilon}\varphi(q)N^{1/2+\varepsilon}w^{1/2+\varepsilon}\int_{0}^{Nw}\left[\log(q)+\left|\log\left(s\right)\right|+\left|\log\left(\left|s-1\right|\right)\right|\right]ds
\]
\[
+\frac{\varphi(q)N^{1/2+\varepsilon}w^{1/2+\varepsilon}}{T}\int_{0}^{Nw}\left(s+1\right)\left[\log\left(qsT\right)^{2}+\log\left(qT/s\right)^{2}\right]ds
\]
and using Proposition \ref{prop:trivial} several times, we get
\[
I_{1}\ll_{\varepsilon}\varphi(q)\log\left(q\right)N^{3/2+\varepsilon}w^{3/2+\varepsilon}\left(1+\left|\log\left(Nw\right)\right|\right)
\]
\[
+\frac{\varphi(q)N^{3/2+\varepsilon}w^{3/2+\varepsilon}\log\left(qT\right)^{2}\left(1+Nw\right)\left(1+\log\left(Nw\right)^{2}\right)}{T}.
\]
\[
\ll_{\varepsilon}\varphi(q)\log\left(q\right)N^{3/2+\varepsilon}w^{3/2+\varepsilon}\left[1+\left|\log\left(Nw\right)\right|+\frac{\left(1+Nw\right)\log\left(qT\right)^{2}\left(1+\log\left(Nw\right)^{2}\right)}{T}\right]
\]
\[
\ll_{\varepsilon}\varphi(q)\log\left(q\right)N^{3/2+\varepsilon}\log\left(N\right)^{2}w^{3/2+\varepsilon}\left(1+\log\left(w\right)^{2}\right)\left[1+\frac{N\left(1+w\right)\log\left(qT\right)^{2}}{T}\right]
\]

\textbf{Bounding $I_{2}.$} 

By Theorem \ref{thm:Explicit M} and recalling again the elementary
bound
\[
\frac{s-s^{\varepsilon}}{\log\left(s\right)}\leq s+s^{\varepsilon},\,s>0,s\neq1
\]
we have
\begin{equation}
I_{2}\ll_{\varepsilon}\varphi\left(q\right)N^{1/2}w^{1/2}\left|\log\left(Nw\right)\right|\left|\log\left(qNw\right)\right|\int_{0}^{Nw}\left[1+s^{\varepsilon}\right]ds\label{eq:I_2 bound}
\end{equation}
\begin{equation}
+\varphi\left(q\right)N^{1/2}w^{1/2}\left|\log\left(wN\right)\right|\left|\log\left(qNw\right)\right|\int_{0}^{Nw}\left[\frac{\left(s+s^{\varepsilon}\right)\left(\left|\log\left(s\right)\right|+1\right)}{U^{1-\varepsilon}}\right]ds\label{eq:I_2 bound-1}
\end{equation}
and again from Proposition \ref{prop:trivial} we get

\[
I_{2}\ll_{\varepsilon}\varphi\left(q\right)N^{3/2+\varepsilon}w^{3/2}\left(1+w^{\varepsilon}\right)\left|\log\left(Nw\right)\right|\left|\log\left(qNw\right)\right|
\]
\[
+\frac{\varphi\left(q\right)N^{5/2}\left|\log\left(Nw\right)\right|\left|\log\left(qNw\right)\right|\left(\left|\log\left(Nw\right)\right|+1\right)w^{1/2+\varepsilon}\left(w^{2-\varepsilon}+1\right)}{U^{1-\varepsilon}}
\]
\[
\ll_{\varepsilon}\varphi\left(q\right)N^{3/2+\varepsilon}\left(\log\left(q\right)^{2}+\log\left(N\right)^{2}\right)\left(1+\log\left(w\right)^{2}\right)w^{3/2}\left(1+w^{\varepsilon}\right)
\]
\[
+\frac{\varphi\left(q\right)N^{5/2}\left(\log\left(q\right)^{3}+\log\left(N\right)^{3}\right)\left(1+\left|\log\left(w\right)\right|^{3}\right)w^{1/2+\varepsilon}\left(w^{2-\varepsilon}+1\right)}{U^{1-\varepsilon}}
\]
and since
\[
\log\left(w\right)^{2}\ll1+\left|\log\left(w\right)\right|^{3},\,\text{\ensuremath{\forall w>0}}
\]
and
\[
\log\left(q\right)^{3}+\log\left(N\right)^{3}\ll\log\left(q\right)^{3}\log\left(N\right)^{3}
\]
we get
\[
I_{2}\ll_{\varepsilon}\varphi\left(q\right)N^{3/2+\varepsilon}\log\left(q\right)^{3}\log\left(N\right)^{3}
\]
\[
\times\left[w^{3/2}\left(1+w^{\varepsilon}\right)\left(1+\log\left(w\right)^{2}\right)+\frac{N^{1-\varepsilon}\left(1+\left|\log\left(w\right)\right|^{3}\right)w^{1/2+\varepsilon}\left(w^{2-\varepsilon}+1\right)}{U^{1-\varepsilon}}\right]
\]

\textbf{Bounding $I_{3}.$} 

We have, by the Cauchy-Schwarz inequality, that
\[
I_{3}\ll_{\varepsilon}\sum_{\chi\neq\chi_{0}}\left(\int_{0}^{Nw}\left|R_{2}(s,U)\right|{}^{2}ds\right)^{1/2}\left(\int_{0}^{Nw}\left|R_{1,\chi}\left(s,T,q\right)\right|^{2}ds\right)^{1/2}
\]
and using the trivial bound $(a_{1}+\dots+a_{n})^{2}\ll_{n}a_{1}^{2}+\dots+a_{n}^{2}$
and Proposition \ref{prop:trivial}, we conclude
\[
\int_{0}^{Nw}R_{1,\chi}\left(s,T,q\right)^{2}ds\ll\int_{0}^{Nw}\left[\log(q)^{2}+\log\left(s\right)^{2}+\left|\log\left(\left|s-1\right|\right)\right|^{2}\right]ds
\]
\[
+\frac{1}{T^{2}}\int_{0}^{Nw}\left(s+1\right)^{2}\left[\log\left(qsT\right)^{4}+\log\left(qT/s\right)^{4}\right]ds
\]
\[
\ll N\left(1+w\right)\log(q)^{2}\log\left(N\right)^{4}\left[\left(1+\log\left(w\right)^{2}\right)+\frac{N^{2}\left(1+w^{2}\right)\log\left(qT\right)^{4}\left(1+\log\left(w\right)^{4}\right)}{T^{2}}\right]
\]
and for $R_{2}\left(s,U\right)$, using the same tools of the previous
bound, we obtain
\[
\int_{0}^{Nw}R_{2}(s,U)^{2}ds\ll_{\varepsilon}\int_{0}^{Nw}\left[1+s^{2\varepsilon}+\left[\frac{\left(s^{2}+s^{2\varepsilon}\right)\left(\log\left(s\right)^{2}+1\right)}{U^{2-2\varepsilon}}\right]\right]ds
\]
\[
\ll_{\varepsilon}N^{1+2\varepsilon}\left(1+w^{1+2\varepsilon}\right)+\frac{N^{3}\left(1+w^{3}\right)\left(1+\log\left(Nw\right)^{2}\right)}{U^{2-2\varepsilon}}
\]
\[
\ll_{\varepsilon}N^{1+2\varepsilon}\log\left(N\right)^{2}\left(1+w^{1+2\varepsilon}\right)\left[1+\frac{N^{2-2\varepsilon}\left(1+w^{2-2\varepsilon}\right)\left(1+\log\left(w\right)^{2}\right)}{U^{2-2\varepsilon}}\right]
\]
and finally, exploiting the subadditivity of the square root, we get
\[
I_{3}\ll_{\varepsilon}\varphi(q)N^{1+\varepsilon}\left(1+w^{1+\varepsilon}\right)\log(q)\log\left(N\right)^{3}\left(1+\left|\log\left(w\right)\right|^{3}\right)
\]
\[
\times\left[1+\frac{N\left(1+w\right)\log\left(qT\right)^{2}}{T}\right]\left[1+\frac{N^{1-\varepsilon}\left(1+w^{1-\varepsilon}\right)}{U^{1-\varepsilon}}\right]
\]
\[
\ll_{\varepsilon}\varphi(q)N^{1+\varepsilon}\log(q)\log\left(N\right)^{3}\left(1+\left|\log\left(w\right)\right|^{3}\right)\left(1+w^{1+\varepsilon}\right)\left[1+\frac{N\log\left(qT\right)^{2}}{T}\right]\left[1+\frac{N^{1-1\varepsilon}}{U^{1-\varepsilon}}\right]
\]

Now, in view of reducing the complexity of the calculations, it is
useful to get a common upper bound for $I_{1}$ and $I_{2}$. Observe
that, since
\[
\left(1+\left|\log\left(A\right)\right|^{k}\right)\ll\left(1+\left|\log\left(A\right)\right|^{n}\right)
\]
if $n\geq k,$ for every $A>0$, then, combining the pieces, we get
\[
I_{1},I_{2}\ll_{\varepsilon}\varphi\left(q\right)N^{3/2+\varepsilon}\log\left(q\right)^{3}\log\left(N\right)^{3}\left(1+w^{3/2+\varepsilon}\right)\left(1+\left|\log\left(w\right)\right|^{3}\right)
\]
\[
\times\left[1+\frac{N\left(1+w\right)\log\left(qT\right)^{2}}{T}+\frac{N^{1-\varepsilon}\left(w^{1-\varepsilon}+1\right)}{U^{1-\varepsilon}}\right].
\]
This concludes the proof
\end{proof}
From the previous explicit formula now we finally obtain a bound for
the weighted average of diagonal twisted EH problem.
\begin{thm}
\label{thm:dEHtwis bound} Fix $0<\varepsilon<1/2$ and assume all
the hypotheses of the previous theorem with $\theta=1-2\varepsilon$.
Then, the following weighted average version of the diagonal twisted
Elliott--Halberstam conjecture holds
\[
\sum_{\underset{{\scriptstyle \left(N,q\right)=1}}{1<q\leq N^{1-2\varepsilon}}}\frac{1}{\varphi\left(q\right)}\left|\sum_{\chi\neq\chi_{0}}\overline{\chi}\left(N\right)\sum_{n\leq N\beta}\sum_{m\leq N\beta-n}\Lambda\left(n\right)\chi\left(n\right)\mu\left(m\right)f\left(\frac{n+m}{N}\right)\right|
\]
\[
\ll_{\varepsilon}N^{2-\varepsilon}E\left(f^{\prime\prime}\right)
\]
where
\[
E\left(f^{\prime\prime}\right):=\int_{0}^{\beta}\left|f^{\prime\prime}\left(w\right)\right|\left(1+\left|\log\left(w\right)\right|^{3}\right)\left(1+w^{2}\right)dw.
\]
and the implicit constant depends only on $\varepsilon$.
\end{thm}

\begin{proof}
From Theorem \ref{thm:Main1} we know that the main average admits
an explicit formula, so the problem boils down to evaluate every single
term of it. We begin by estimating the main term. We have, fixing
$\varepsilon>0$, that
\[
\left|\sum_{\rho:\left|\gamma\right|\leq U}\frac{N^{\rho}\Gamma\left(\rho\right)}{\zeta^{\prime}(\rho)}\sum_{\chi\neq\chi_{0}}\overline{\chi}\left(N\right)\sum_{\rho_{\chi}:\left|\gamma_{\chi}\right|\leq T}\frac{N^{\rho_{\chi}}\Gamma\left(\rho_{\chi}\right)}{\Gamma\left(\rho_{\chi}+\rho+2\right)}\int_{0}^{\beta}f^{\prime\prime}\left(w\right)w^{\rho+\rho_{\chi}+1}dw\right|
\]
\begin{equation}
\leq N\varphi\left(q\right)\int_{0}^{\beta}\left|f^{\prime\prime}\left(w\right)\right|w^{2}dw\max_{\chi\mod q}\sum_{\rho:\left|\gamma\right|\leq U}\frac{1}{\left|\zeta^{\prime}(\rho)\right|}\sum_{\rho_{\chi}:\left|\gamma_{\chi}\right|\leq T}\left|\frac{\Gamma\left(\rho\right)\Gamma\left(\rho_{\chi}\right)}{\Gamma\left(\rho_{\chi}+\rho+2\right)}\right|.\label{eq:esti chi mod q}
\end{equation}
Now, by Theorem \ref{thm:dobule series =00005Cchi} we observe that
the double series converges absolutely, since $f\left(z\right):=\frac{1}{\left|\zeta^{\prime}(z)\right|},\,g\left(z\right)\equiv1$
verify the hypotheses, by Conjecture \ref{conj:simple zeros}, so
we may let $T,U\rightarrow+\infty$. As for the dependence on $q$,
we observe, starting considering $\gamma_{\chi}>4$, that
\[
\sum_{\rho}\frac{1}{\left|\zeta^{\prime}(\rho)\right|}\sum_{\rho_{\chi}:\gamma_{\chi}>4}\left|\frac{\Gamma\left(\rho\right)\Gamma\left(\rho_{\chi}\right)}{\Gamma\left(\rho_{\chi}+\rho+2\right)}\right|=\sum_{\rho}\frac{1}{\left|\zeta^{\prime}(\rho)\right|}\sum_{k\geq2}\sum_{\rho_{\chi}:2^{k}<\gamma_{\chi}\leq2^{k+1}}\left|\frac{\Gamma\left(\rho\right)\Gamma\left(\rho_{\chi}\right)}{\Gamma\left(\rho_{\chi}+\rho+2\right)}\right|
\]
and since, if $T\geq4$ and $q>1$, we have, by (\ref{eq:N(T,=00005Cchi)}),
that
\[
\sum_{\rho_{\chi}:2^{k}<\gamma_{\chi}\leq2^{k+1}}\left|\frac{\Gamma\left(\rho\right)\Gamma\left(\rho_{\chi}\right)}{\Gamma\left(\rho_{\chi}+\rho+2\right)}\right|\ll\left(N\left(2^{k+1},\chi\right)-N\left(2^{k},\chi\right)\right)\max_{\rho_{\chi}:2^{k}<\gamma_{\chi}\leq2^{k+1}}\left|\frac{\Gamma\left(\frac{1}{2}+i\gamma\right)\Gamma\left(\frac{1}{2}+i\gamma_{\chi}\right)}{\Gamma\left(3+i\left(\gamma+\gamma_{\chi}\right)\right)}\right|.
\]
Now, by Proposition \ref{prop:gamma est} we deduce that
\[
\left|\frac{\Gamma\left(\frac{1}{2}+i\gamma\right)\Gamma\left(\frac{1}{2}+i\gamma_{\chi}\right)}{\Gamma\left(3+i\left(\gamma+\gamma_{\chi}\right)\right)}\right|\ll\begin{cases}
\frac{e^{\pi\left(\left|\gamma+\gamma_{\chi}\right|-\left|\gamma\right|-\gamma_{\chi}\right)/2}}{\left|\gamma+\gamma_{\chi}\right|^{5/2}}, & \left|\gamma+\gamma_{\chi}\right|>1\\
e^{-\pi\left(\left|\gamma\right|+\gamma_{\chi}\right)/2}, & \left|\gamma+\gamma_{\chi}\right|\leq1
\end{cases}
\]
hence it is enough to consider the ``worst'' case $\gamma,\gamma_{\chi}>0$.
In such case, by (\ref{eq:N(T,=00005Cchi)}), we get
\[
\sum_{\rho:\gamma>0}\frac{1}{\left|\zeta^{\prime}(\rho)\right|}\sum_{k\geq2}\sum_{\rho_{\chi}:2^{k}<\gamma_{\chi}\leq2^{k+1}}\left|\frac{\Gamma\left(\rho\right)\Gamma\left(\rho_{\chi}\right)}{\Gamma\left(\rho_{\chi}+\rho+2\right)}\right|\ll\log\left(q\right)\sum_{\rho:\gamma>0}\frac{1}{\left|\zeta^{\prime}(\rho)\right|}\sum_{k\geq2}\frac{2^{k}\log\left(2^{k}\right)}{\left(\gamma+2^{k}\right)^{5/2}}
\]
\[
\ll\log\left(q\right)\sum_{\rho:\gamma>0}\frac{1}{\left|\zeta^{\prime}(\rho)\right|\gamma^{5/4}}\sum_{k\geq2}\log\left(k\right)2^{-k/4}\ll\log\left(q\right)
\]
the remaining cases, that is $\left|\gamma_{\chi}\right|\leq4$ or
$\gamma_{\chi}<-4$ or $\gamma<0$ can be treated similarly. Hence,
the main term can be bounded by
\[
\ll N\varphi\left(q\right)\log\left(q\right)\int_{0}^{\beta}\left|f^{\prime\prime}\left(w\right)\right|w^{2}dw
\]
where the implicit constant is absolute. It remain to observe that,
if $T,U\rightarrow+\infty$, the error terms are
\[
\lim_{T,U\rightarrow+\infty}\mathcal{E}_{1}\left(N,w,q,U,T\right)=\varphi\left(q\right)N^{3/2+\varepsilon}\log\left(q\right)^{3}\log\left(N\right)^{3}\left(1+w^{3/2+\varepsilon}\right)\left(1+\left|\log\left(w\right)\right|^{3}\right)
\]
\[
\lim_{T,U\rightarrow+\infty}\mathcal{E}_{2}\left(N,w,q,U,T\right)=\varphi(q)N^{1+\varepsilon}\log(q)\log\left(N\right)^{3}\left(1+\left|\log\left(w\right)\right|^{3}\right)\left(1+w^{2}\right)
\]
Therefore, these terms produce the error
\[
E_{\max}:=\varphi(q)N^{1/2+\varepsilon}\log\left(q\right)^{3}\log\left(N\right)^{3}\int_{0}^{\beta}\left|f^{\prime\prime}\left(w\right)\right|\left(1+\left|\log\left(w\right)\right|^{3}\right)\left(1+w^{2}\right)dw
\]
hence we get, taking $\theta=1-2\varepsilon$,
\[
\sum_{\underset{{\scriptstyle \left(N,q\right)=1}}{1<q\leq N^{1-2\varepsilon}}}\frac{1}{\varphi\left(q\right)}\left|\sum_{\chi\neq\chi_{0}}\overline{\chi}\left(N\right)\sum_{n\leq N\beta}\sum_{m\leq N\beta-n}\Lambda\left(n\right)\chi\left(n\right)\mu\left(m\right)f\left(\frac{n+m}{N}\right)\right|
\]
\[
\ll_{\varepsilon}N\int_{0}^{\beta}\left|f^{\prime\prime}\left(w\right)\right|w^{2}dw\sum_{q\leq N^{1-2\varepsilon}}\log\left(q\right)
\]
\[
+N^{1/2+\varepsilon}\int_{0}^{\beta}\left|f^{\prime\prime}\left(w\right)\right|\left(1+\left|\log\left(w\right)\right|^{3}\right)\left(1+w^{2}\right)dw\sum_{q\leq N^{1-2\varepsilon}}\log\left(q\right)^{3}
\]
\[
\ll_{\varepsilon}\left(N^{2-2\varepsilon}\log\left(N\right)+\log\left(N\right)^{3}N^{3/2}\right)E\left(f^{\prime\prime}\right)\ll_{\varepsilon}N^{2-\varepsilon}E\left(f^{\prime\prime}\right)
\]
where
\[
E\left(f^{\prime\prime}\right):=\int_{0}^{\beta}\left|f^{\prime\prime}\left(w\right)\right|\left(1+\left|\log\left(w\right)\right|^{3}\right)\left(1+w^{2}\right)dw.
\]
So we finally get
\[
\sum_{\underset{{\scriptstyle \left(N,q\right)=1}}{1<q\leq N^{1-2\varepsilon}}}\frac{1}{\varphi\left(q\right)}\left|\sum_{\chi\neq\chi_{0}}\overline{\chi}\left(N\right)\sum_{n\leq N\beta}\sum_{m\leq N\beta-n}\Lambda\left(n\right)\chi\left(n\right)\mu\left(m\right)f\left(\frac{n+m}{N}\right)\right|
\]
\[
\ll_{\varepsilon}N^{2-\varepsilon}E\left(f^{\prime\prime}\right).
\]
\end{proof}

\subsection{Weights in H\"older-Zygmund class}

As we said in the previous sections, the hypothesis $f_{\vert(0,\beta)}\in W^{2,1}\left(0,\beta\right)$
is, morally, minimal if we want to use Theorem \ref{thm:MAINCGZ},
but this is no longer the case for the discrete version of that theorem.
In the latter case, we are able to work with more general functions,
and we will show that the regularity offered by the H\"older-Zygmund
class is sufficient. Again, we first need an explicit formula for
the weighted average.
\begin{thm}
\label{thm:explicit discrete}Fix $0<\varepsilon<1/2$. Assume GRH
and the Conjecture \ref{conj:simple zeros}. Let $N\geq4$, $T,U\geq4$,
$q\in\mathbb{N},\,1<q<N^{\theta},\,0<\theta<1.$ Moreover, let $\beta>0$
and $f:\mathbb{R}\rightarrow\mathbb{C}$ satisfy the hypotheses of
Theorem \ref{thm:MainDiscrete}. Fix $\delta\in\left[1,2\right)$
and assume $f\in\mathcal{C}^{\delta}\left(\mathbb{R}\right)$. Then
\[
\frac{1}{\varphi\left(q\right)}\sum_{\chi\neq\chi_{0}}\overline{\chi}\left(N\right)\sum_{n\leq N\beta}\sum_{m\leq N\beta-n}\Lambda\left(n\right)\chi\left(n\right)\mu\left(m\right)f\left(\frac{n+m}{N}\right)
\]
\[
=-\frac{1}{\varphi\left(q\right)}\sum_{\rho:\left|\gamma\right|\leq U}\frac{\Gamma\left(\rho\right)}{\zeta^{\prime}(\rho)}\sum_{\chi\neq\chi_{0}}\overline{\chi}\left(N\right)\sum_{\rho_{\chi}:\left|\gamma_{\chi}\right|\leq T}\frac{\Gamma\left(\rho_{\chi}\right)}{\Gamma\left(\rho_{\chi}+\rho+2\right)}
\]
\[
\times\sum_{k\leq N\beta-1}\Delta_{1/N}^{2}\left(f,\frac{k}{N}\right)k^{\rho+\rho_{\chi}+1}
\]
\[
+O_{f,\beta,\varepsilon,\theta}\left(N^{5/2-\delta}\log\left(qT\right)^{2}+\frac{N^{3-\delta}\log\left(qNT\right)^{2}}{T}\right.
\]
\[
\left.+\log\left(q\right)^{3}N^{5/2-\delta+\varepsilon}\log\left(N\right)^{3}\left[1+\frac{N\log\left(qT\right)^{2}}{T}\right]\left[1+\frac{N^{1-\varepsilon}}{U^{1-\varepsilon}}\right]\right)
\]
\end{thm}

\begin{proof}
We fix $\delta\in\left[1,2\right)$ and $f\in\mathcal{C}^{\delta}\left(\mathbb{R}\right)$,
then we have
\[
\Delta_{1/N}^{2}\left(f,\frac{k}{N}\right)\ll_{f}N^{-\delta}
\]
by Theorem \ref{thm:finitediff bound}. From this fact and from Corollary
\ref{cor:discrete}, we have
\[
\frac{1}{\varphi\left(q\right)}\sum_{\chi\neq\chi_{0}}\overline{\chi}\left(N\right)\sum_{n\leq N\beta}\sum_{m\leq N\beta-n}\Lambda\left(n\right)\chi\left(n\right)\mu\left(m\right)f\left(\frac{n+m}{N}\right)
\]
\[
=\frac{1}{\varphi\left(q\right)}\sum_{k\leq N\beta-1}\Delta_{1/N}^{2}\left(f,\frac{k}{N}\right)\sum_{\chi\neq\chi_{0}}\overline{\chi}\left(N\right)\int_{0}^{k}\psi\left(s,\chi\right)M\left(k-s\right)ds
\]
\[
+O_{f,\beta}\left(\frac{N^{-\delta}}{\varphi(q)}\sum_{k\leq N\beta-1}\sum_{\chi\neq\chi_{0}}\left|\sum_{u=0}^{k}\psi(u,\chi)\mu(k-u)\right|\right).
\]
Using the explicit formulas and the calculations of Theorem \ref{thm:Main1}
and the same notation, we get 
\[
\int_{0}^{k}M\left(s\right)\sum_{\chi\neq\chi_{0}}\overline{\chi}\left(N\right)\psi\left(k-s,\chi\right)ds=-\sum_{\rho:\left|\gamma\right|\leq U}\frac{\Gamma\left(\rho\right)}{\zeta^{\prime}(\rho)}\sum_{\chi\neq\chi_{0}}\overline{\chi}\left(N\right)\sum_{\rho_{\chi}:\left|\gamma_{\chi}\right|\leq T}\frac{\Gamma\left(\rho_{\chi}\right)k^{\rho+\rho_{\chi}+1}}{\Gamma\left(\rho_{\chi}+\rho+2\right)}
\]
\[
+O_{\varepsilon}\left(I_{1}\left(k/N,N,q,T\right)+I_{2}\left(k/N,N,q,U\right)+I_{3}\left(k/N,N,q,T,U\right)\right).
\]
So we obtain
\[
\frac{1}{\varphi\left(q\right)}\sum_{\chi\neq\chi_{0}}\overline{\chi}\left(N\right)\sum_{n\leq N\beta}\sum_{m\leq N\beta-n}\Lambda\left(n\right)\chi\left(n\right)\mu\left(m\right)f\left(\frac{n+m}{N}\right)
\]
\[
=-\frac{1}{\varphi\left(q\right)}\sum_{\rho:\left|\gamma\right|\leq U}\frac{\Gamma\left(\rho\right)}{\zeta^{\prime}(\rho)}\sum_{\chi\neq\chi_{0}}\overline{\chi}\left(N\right)\sum_{\rho_{\chi}:\left|\gamma_{\chi}\right|\leq T}\frac{\Gamma\left(\rho_{\chi}\right)}{\Gamma\left(\rho_{\chi}+\rho+2\right)}\sum_{k\leq N\beta-1}\Delta_{1/N}^{2}\left(f,\frac{k}{N}\right)k^{\rho+\rho_{\chi}+1}
\]
\[
+O_{f,\varepsilon}\left(\frac{N^{-\delta}}{\varphi\left(q\right)}\sum_{k\leq N\beta-1}\left[I_{1}\left(k/N,N,q,T\right)+I_{2}\left(k/N,N,q,U\right)+I_{3}\left(k/N,N,q,T,U\right)\right]\right)
\]
\[
+O_{f,\beta}\left(\frac{N^{-\delta}}{\varphi(q)}\sum_{k\leq N\beta-1}\sum_{\chi\neq\chi_{0}}\left|\sum_{u=0}^{k}\psi\left(u,\chi\right)\mu\left(k-u\right)\right|\right).
\]

We begin by estimating $I_{1},I_{2}$. Since we proved in Theorem
\ref{thm:Main1} that
\[
I_{1},I_{2}\ll\varphi\left(q\right)N^{3/2+\varepsilon}\log\left(q\right)^{3}\log\left(N\right)^{3}\left(1+w^{3/2+\varepsilon}\right)\left(1+\left|\log\left(w\right)\right|^{3}\right)
\]
\[
\times\left[1+\frac{N\left(1+w\right)\log\left(qT\right)^{2}}{T}+\frac{N^{1-\varepsilon}\left(w^{1-\varepsilon}+1\right)}{U^{1-\varepsilon}}\right]
\]
but, in this case, the convolution is from $0$ to $k$ instead of
$0$ to $Nw$, hence we have to consider
\[
\log\left(q\right)^{3}\varphi\left(q\right)N^{-\delta}\sum_{k\leq N\beta-1}k^{3/2+\varepsilon}\log\left(k\right)^{3}\left[1+\frac{k\log\left(qT\right)^{2}}{T}+\frac{k^{1-\varepsilon}}{U^{1-\varepsilon}}\right]
\]
\[
\ll_{\beta}\log\left(q\right)^{3}\varphi\left(q\right)N^{5/2-\delta+\varepsilon}\log\left(N\right)^{3}\left[1+\frac{N\log\left(qT\right)^{2}}{T}+\frac{N^{1-\varepsilon}}{U^{1-\varepsilon}}\right].
\]
In the same spirit, taking $I_{3}$, we have to consider
\[
\log(q)\varphi\left(q\right)N^{-\delta}\sum_{k\leq N\beta-1}k^{1+\varepsilon}\log\left(k\right)^{3}\left[1+\frac{k\log\left(qT\right)^{2}}{T}\right]\left[1+\frac{k^{1-\varepsilon}}{U^{1-\varepsilon}}\right]
\]
\[
\ll_{\beta}\varphi\left(q\right)\log(q)N^{2+\varepsilon-\delta}\log\left(N\right)^{3}\left[1+\frac{N\log\left(qT\right)^{2}}{T}\right]\left[1+\frac{N^{1-\varepsilon}}{U^{1-\varepsilon}}\right].
\]

Now, we consider the second error term. We have
\[
\frac{N^{-\delta}}{\varphi(q)}\sum_{k\leq N\beta-1}\sum_{\chi\neq\chi_{0}}\left|\sum_{u=0}^{k}\psi\left(u,\chi\right)\mu\left(k-u\right)\right|.
\]
Inserting the explicit formula of $\psi\left(u,\chi\right)$ proved
in Theorem \ref{thm:Explicit psi}, we get, recalling that $q<N$,
\[
\frac{N^{-\delta}}{\varphi(q)}\sum_{k\leq N\beta-1}\sum_{\chi\neq\chi_{0}}\left|\sum_{\rho_{\chi}:\left|\gamma_{\chi}\right|\leq T}\frac{1}{\rho_{\chi}}\sum_{u=1}^{k-1}u^{\rho_{\chi}}\mu(k-u)\right|
\]
\[
+O_{\beta,\theta}\left(N^{2-\delta}\log\left(N\right)+N^{-\delta}\sum_{k\leq N\beta-1}\sum_{u=1}^{k-1}\frac{u\log\left(quT\right)^{2}}{T}\right).
\]
Clearly
\[
N^{-\delta}\sum_{k\leq N\beta-1}\sum_{u=1}^{k-1}\frac{u\log\left(quT\right)^{2}}{T}\ll\frac{N^{3-\delta}\log\left(qNT\right)^{2}}{T}
\]
anf from the main term we get
\[
\frac{N^{-\delta}}{\varphi(q)}\sum_{k\leq N\beta-1}\sum_{\chi\neq\chi_{0}}\left|\sum_{\rho_{\chi}:\left|\gamma_{\chi}\right|\leq T}\frac{1}{\rho_{\chi}}\sum_{u=1}^{k-1}u^{\rho_{\chi}}\mu(k-u)\right|
\]
\[
\ll\frac{N^{-\delta}}{\varphi(q)}\sum_{k\leq N\beta-1}k^{3/2}\sum_{\chi\neq\chi_{0}}\sum_{\rho_{\chi}:\left|\gamma_{\chi}\right|\leq T}\frac{1}{\left|\rho_{\chi}\right|}\ll N^{5/2-\delta}\log\left(qT\right)^{2}
\]
from the classical results
\[
\sum_{\rho_{\chi}:\left|\gamma_{\chi}\right|\leq T}\frac{1}{\left|\rho_{\chi}\right|}\ll\log\left(qT\right)^{2}
\]
which can be proved using (\ref{eq:N(T,=00005Cchi)}). The claim follows
by combining the previous estimates.
\end{proof}
Now we are able to prove the main theorem of this subsection.
\begin{thm}
\label{thm:dEhtwist discrete}Under the hypotheses of the previous
theorem and taking $\theta=\delta-1-2\varepsilon$ with a fixed $0<\varepsilon<\left(\delta-1\right)/2$
if $\delta\in\left[\frac{3}{2},2\right)$ and for $\theta=\frac{1}{2}-2\varepsilon$
with a fixed $0<\varepsilon<1/4$ if $\delta\in\left[1,3/2\right)$,
we get
\[
\sum_{1<q\leq N^{\theta}}\frac{1}{\varphi\left(q\right)}\left|\sum_{\chi\neq\chi_{0}}\overline{\chi}\left(N\right)\sum_{n\leq N\beta}\sum_{m\leq N\beta-n}\Lambda\left(n\right)\chi\left(n\right)\mu\left(m\right)f\left(\frac{n+m}{N}\right)\right|
\]
\[
\ll_{f,\beta,\varepsilon,\delta}N^{2-\varepsilon}.
\]
\end{thm}

\begin{proof}
We begin by considering the case $\delta\in\left[\frac{3}{2},2\right)$.
By the previous explicit formula, the main term and the regularity
of the weight,
\[
-\frac{1}{\varphi\left(q\right)}\sum_{\rho:\left|\gamma\right|\leq U}\frac{\Gamma\left(\rho\right)}{\zeta^{\prime}(\rho)}\sum_{\chi\neq\chi_{0}}\overline{\chi}\left(N\right)\sum_{\rho_{\chi}:\left|\gamma_{\chi}\right|\leq T}\frac{\Gamma\left(\rho_{\chi}\right)}{\Gamma\left(\rho_{\chi}+\rho+2\right)}\sum_{k\leq N\beta-1}\Delta_{1/N}^{2}\left(f,\frac{k}{N}\right)k^{\rho+\rho_{\chi}+1}
\]
\[
\ll_{f,\beta}\frac{N^{3-\delta}}{\varphi(q)}\sum_{\rho:\left|\gamma\right|\leq U}\left|\frac{\Gamma\left(\rho\right)}{\zeta^{\prime}(\rho)}\right|\sum_{\chi\neq\chi_{0}}\sum_{\rho_{\chi}:\left|\gamma_{\chi}\right|\leq T}\left|\frac{\Gamma\left(\rho_{\chi}\right)}{\Gamma\left(\rho_{\chi}+\rho+2\right)}\right|.
\]
Now, arguing as in Theorem \ref{thm:dEHtwis bound} we have that the
double series converges absolutely and it is $\ll\log\left(q\right)$,
so we have
\[
-\frac{1}{\varphi\left(q\right)}\sum_{\rho:\left|\gamma\right|\leq U}\frac{\Gamma\left(\rho\right)}{\zeta^{\prime}(\rho)}\sum_{\chi\neq\chi_{0}}\overline{\chi}\left(N\right)\sum_{\rho_{\chi}:\left|\gamma_{\chi}\right|\leq T}\frac{\Gamma\left(\rho_{\chi}\right)}{\Gamma\left(\rho_{\chi}+\rho+2\right)}\sum_{k\leq N\beta-1}\Delta_{1/N}^{2}\left(f,\frac{k}{N}\right)k^{\rho+\rho_{\chi}+1}
\]
\[
\ll_{f,\beta}N^{3-\delta}\log\left(q\right)\ll_{f,\beta,\delta}N^{3-\delta}\log\left(N\right)
\]

Now we consider the error term. Taking $U<T,\,U\asymp T\asymp N,$we
obtain

\[
E_{\max}^{*}:=N^{5/2-\delta}\log\left(qT\right)^{2}+\frac{N^{3-\delta}\log\left(qNT\right)^{2}}{T}
\]
\[
+\log\left(q\right)^{3}N^{5/2-\delta+\varepsilon}\log\left(N\right)^{3}\left[1+\frac{N\log\left(qT\right)^{2}}{T}\right]\left[1+\frac{N^{1-\varepsilon}}{U^{1-\varepsilon}}\right]
\]
\[
\ll_{\varepsilon}\log\left(q\right)^{5}N^{5/2-\delta+\varepsilon}\log\left(N\right)^{5}.
\]
 Combining all the pieces and taking $\theta=\delta-1-2\varepsilon$,
we get
\[
\sum_{1<q\leq N^{\delta-1-2\varepsilon}}\frac{1}{\varphi\left(q\right)}\left|\sum_{\chi\neq\chi_{0}}\overline{\chi}\left(N\right)\sum_{n\leq N\beta}\sum_{m\leq N\beta-n}\Lambda\left(n\right)\chi\left(n\right)\mu\left(m\right)f\left(\frac{n+m}{N}\right)\right|
\]
\[
\ll_{f,\beta,\varepsilon}N^{2-2\varepsilon}\log\left(N\right)+N^{5/2-\delta+\varepsilon}\log\left(N\right)^{5}\sum_{q\leq N^{\delta-1-2\varepsilon}}\log\left(q\right)^{5}\ll_{f,\beta,\varepsilon,\delta}N^{2-\varepsilon}.
\]

Now, assume that $\delta\in\left[1,\frac{3}{2}\right)$. As regards
the error term, taking again $U<T,\,U\asymp T\asymp N,$we obtain
the same bound
\[
E_{\max}^{*}\ll\log\left(q\right)^{5}N^{5/2-\delta+\varepsilon}\log\left(N\right)^{5}.
\]

So now we focus on the main term, which can be treated in a different
way with respect to the previous case, since a direct use of the absolute
convergence of the series leads to a too large error in $N$. Using
summation by parts and recalling the property (\ref{eq:prod fow diff}),
which means in this case
\[
\Delta_{1/N}^{2}\left(f,\frac{k}{N}\right)=\Delta_{1/N}\left(f,\frac{k+1}{N}\right)-\Delta_{1/N}\left(f,\frac{k}{N}\right)
\]

we get
\[
\sum_{k\leq N\beta-1}\Delta_{1/N}^{2}\left(f,\frac{k}{N}\right)k^{\rho+\rho_{\chi}+1}=\Delta_{1/N}\left(f,\beta\right)\left(N\beta-1\right)^{\rho+\rho_{\chi}+1}
\]
\[
-\sum_{k\leq N\beta-1}\Delta_{1/N}\left(f,\frac{k}{N}\right)\left(k^{\rho+\rho_{\chi}+1}-\left(k-1\right)^{\rho+\rho_{\chi}+1}\right).
\]
We recall that if $f\in\mathcal{C}^{1}\left(\mathbb{R}\right)$, then
$\Delta_{1/N}\left(f,x\right)\ll N^{-1}\log\left(N\right)$, by (\ref{eq:LogLip}),
so clearly this bound can be used for every fixed $\delta\in\left[1,\frac{3}{2}\right)$.
Hence, we have
\[
\Delta_{1/N}\left(f,\beta\right)\left(N\beta-1\right)^{\rho+\rho_{\chi}+1}\ll_{f,\beta}N\log\left(N\right)
\]
so it remains to bound the sum. Note that
\[
\sum_{k\leq N\beta-1}\Delta_{1/N}\left(f,\frac{k}{N}\right)\left(k^{\rho+\rho_{\chi}+1}-\left(k-1\right)^{\rho+\rho_{\chi}+1}\right)
\]
\[
=\sum_{k\leq N\beta-1}\Delta_{1/N}\left(f,\frac{k}{N}\right)\left(\rho+\rho_{\chi}+1\right)\int_{k-1}^{k}v^{\rho+\rho_{\chi}}dv
\]
\[
=\sum_{k\leq N\beta-1}\Delta_{1/N}\left(f,\frac{k}{N}\right)\left(\rho+\rho_{\chi}+1\right)\int_{0}^{1}\left(k-1+t\right)^{\rho+\rho_{\chi}}dt
\]
\[
\ll_{f,\beta}\frac{\left|\rho+\rho_{\chi}+1\right|}{N}\log\left(N\right)\sum_{1<k\leq N\beta-1}k\ll_{f,\beta}N\log\left(N\right)\left|\rho+\rho_{\chi}+1\right|
\]

so we have to estimate
\[
\frac{N\log\left(N\right)}{\varphi\left(q\right)}\sum_{\rho:\left|\gamma\right|\leq U}\frac{1}{\left|\zeta^{\prime}(\rho)\right|}\sum_{\chi\neq\chi_{0}}\sum_{\rho_{\chi}:\left|\gamma_{\chi}\right|\leq T}\left|\frac{\Gamma\left(\rho\right)\Gamma\left(\rho_{\chi}\right)}{\Gamma\left(\rho_{\chi}+\rho+1\right)}\right|
\]
and in this case we cannot use the absolute convergence of the series.
We start considering $\left|\gamma+\gamma_{\chi}\right|<1$. Then,
recalling that
\[
N_{\chi}\left(t+1,\chi\right)-N_{\chi}\left(t,\chi\right)\ll\log\left(q\left(\left|t\right|+2\right)\right)
\]
which is, essentially, a consequence of (\ref{eq:N(T,=00005Cchi)})
(see \cite{MV2006}, Theorem $10.17$), we have, from Proposition
\ref{prop:gamma est}, that
\[
\sum_{\rho:\left|\gamma\right|\leq U}\frac{1}{\left|\zeta^{\prime}(\rho)\right|}\sum_{\rho_{\chi}:\left|\gamma_{\chi}\right|\leq T,\,\left|\gamma+\gamma_{\chi}\right|<1}\left|\frac{\Gamma\left(\rho\right)\Gamma\left(\rho_{\chi}\right)}{\Gamma\left(\rho_{\chi}+\rho+1\right)}\right|
\]
\[
\ll\sum_{\rho:\left|\gamma\right|\leq U}\frac{e^{-\pi\left|\gamma\right|/2}}{\left|\zeta^{\prime}(\rho)\right|}\sum_{\rho_{\chi}:\left|\gamma_{\chi}\right|\leq T,\,\left|\gamma+\gamma_{\chi}\right|<1}e^{-\pi\left|\gamma_{\chi}\right|/2}
\]
\[
\ll\sum_{\rho:\left|\gamma\right|\leq U}\frac{e^{-\pi\left|\gamma\right|/2}}{\left|\zeta^{\prime}(\rho)\right|}\#\left\{ \rho_{\chi}:\gamma_{\chi}\in\left(-\gamma-1,-\gamma+1\right)\right\} 
\]
\[
\ll\log\left(q\right)\sum_{\rho:\left|\gamma\right|\leq U}\frac{e^{-\pi\left|\gamma\right|/2}}{\left|\zeta^{\prime}(\rho)\right|}\log\left(\left|\gamma\right|\right)\ll\log(q).
\]
Now, assume $\left|\gamma+\gamma_{\chi}\right|\geq1$. Then, again
from Proposition \ref{prop:gamma est},
\[
\sum_{\rho:\left|\gamma\right|\leq U}\frac{1}{\left|\zeta^{\prime}(\rho)\right|}\sum_{\rho_{\chi}:\left|\gamma_{\chi}\right|\leq T,\,\left|\gamma+\gamma_{\chi}\right|\geq1}\left|\frac{\Gamma\left(\rho\right)\Gamma\left(\rho_{\chi}\right)}{\Gamma\left(\rho_{\chi}+\rho+1\right)}\right|
\]
\[
\ll\sum_{\rho:\left|\gamma\right|\leq U}\frac{1}{\left|\zeta^{\prime}(\rho)\right|}\sum_{\rho_{\chi}:\left|\gamma_{\chi}\right|\leq T,\,\left|\gamma+\gamma_{\chi}\right|\geq1}\frac{e^{-\pi\left(\left|\gamma\right|+\left|\gamma_{\chi}\right|-\left|\gamma+\gamma_{\chi}\right|\right)/2}}{\left|\gamma+\gamma_{\chi}\right|^{3/2}}.
\]
If $\gamma\gamma_{\chi}>0$, then it is enough to consider
\[
\sum_{\rho:0<\gamma\leq U}\frac{1}{\left|\zeta^{\prime}(\rho)\right|}\sum_{\rho_{\chi}:0<\gamma_{\chi}\leq T,\,\gamma+\gamma_{\chi}\geq1}\frac{1}{\left(\gamma+\gamma_{\chi}\right)^{3/2}}
\]
\[
\ll\sum_{\rho:0<\gamma\leq U}\frac{1}{\left|\zeta^{\prime}(\rho)\right|\gamma^{3/4}}\sum_{\rho_{\chi}:0<\gamma_{\chi}\leq T,\,\gamma+\gamma_{\chi}\geq1}\frac{1}{\gamma_{\chi}^{3/4}}
\]
\[
\ll\left(TU\right)^{1/4}\log\left(qT\right)^{2}\log\left(U\right)^{1/2}
\]
By (\ref{eq:N(T)}) and Conjecture \ref{conj:simple zeros}. Assume
now $\gamma\gamma_{\chi}<0$ and let $\gamma>0,\gamma_{\chi}<0$.
We have
\[
\sum_{\rho:0<\gamma\leq U}\frac{1}{\left|\zeta^{\prime}(\rho)\right|}\sum_{\rho_{\chi}:0<\gamma_{\chi}\leq T,\,\gamma-\gamma_{\chi}\geq1}\frac{e^{-\pi\gamma_{\chi}}}{\left(\gamma-\gamma_{\chi}\right)^{3/2}}
\]
\[
+\sum_{\rho:0<\gamma\leq U}\frac{e^{-\pi\gamma}}{\left|\zeta^{\prime}(\rho)\right|}\sum_{\rho_{\chi}:0<\gamma_{\chi}\leq T,\,\gamma-\gamma_{\chi}\leq-1}\frac{1}{\left(\gamma_{\chi}-\gamma\right)^{3/2}}
\]
then note that 
\[
\,\gamma-\gamma_{\chi}\geq1\Rightarrow\gamma_{\chi}\leq\gamma-1\leq U-1<T
\]
hence we can consider
\[
\sum_{\rho_{\chi}:0<\gamma_{\chi}\leq T,\,\gamma-\gamma_{\chi}\leq-1}\frac{e^{-\pi\gamma_{\chi}}}{\left(\gamma-\gamma_{\chi}\right)^{3/2}}=\sum_{\rho_{\chi}:0<\gamma_{\chi}\leq\gamma-1}\frac{e^{-\pi\gamma_{\chi}}}{\left(\gamma-\gamma_{\chi}\right)^{3/2}}
\]
now if $\gamma_{\chi}<\gamma/2$ then
\[
\sum_{\rho_{\chi}:0<\gamma_{\chi}\leq\gamma/2}\frac{e^{-\pi\gamma_{\chi}}}{\left(\gamma-\gamma_{\chi}\right)^{3/2}}\ll\frac{\log\left(q\right)}{\gamma^{3/2}}
\]
and if $\gamma/2\leq\gamma_{\chi}<\gamma-1$ then
\[
\sum_{\rho_{\chi}:\gamma/2<\gamma_{\chi}\leq\gamma-1}\frac{e^{-\pi\gamma_{\chi}}}{\left(\gamma-\gamma_{\chi}\right)^{3/2}}\ll e^{-\pi\gamma/2}\gamma\log\left(q\gamma\right)
\]
then
\[
\sum_{\rho:0<\gamma\leq U}\frac{1}{\left|\zeta^{\prime}(\rho)\right|}\sum_{\rho_{\chi}:0<\gamma_{\chi}\leq T,\,\gamma-\gamma_{\chi}\geq1}\frac{e^{-\pi\gamma_{\chi}}}{\left(\gamma-\gamma_{\chi}\right)^{3/2}}
\]
\[
\ll\log\left(q\right)\left[\sum_{\rho:0<\gamma\leq U}\frac{1}{\left|\zeta^{\prime}(\rho)\right|\gamma^{3/2}}+\sum_{\rho:0<\gamma\leq U}\frac{\gamma\log\left(\gamma\right)}{\left|\zeta^{\prime}(\rho)\right|}\right]\ll\log\left(q\right)
\]
due to the absolute convergence of the series. For the second sum
we clearly have
\[
\sum_{\rho_{\chi}:0<\gamma_{\chi}\leq T,\,\gamma-\gamma_{\chi}\leq-1}\frac{1}{\left(\gamma_{\chi}-\gamma\right)^{3/2}}=\sum_{\rho_{\chi}:\gamma+1\leq\gamma_{\chi}\leq T}\frac{1}{\left(\gamma_{\chi}-\gamma\right)^{3/2}}
\]
\[
\ll\sum_{m\in\left[1,T-\gamma\right]\cap\mathbb{N}}\sum_{\gamma_{\chi}\in\left[\gamma+m,\gamma+m+1\right)}\frac{1}{\left(\gamma_{\chi}-\gamma\right)^{3/2}}
\]
\[
\ll\sum_{m\geq1}\frac{\log\left(q\left(\gamma+m+2\right)\right)}{m^{3/2}}\ll\log\left(q\right)+\log\left(\gamma\right)
\]
hence
\[
\sum_{\rho:0<\gamma\leq U}\frac{e^{-\pi\gamma}}{\left|\zeta^{\prime}(\rho)\right|}\sum_{\rho_{\chi}:0<\gamma_{\chi}\leq T,\,\gamma-\gamma_{\chi}\leq-1}\frac{1}{\left(\gamma_{\chi}-\gamma\right)^{3/2}}\ll\log\left(q\right).
\]
It remains to control what happens if $\gamma_{\chi}=0$, which is
a case that we cannot exclude. We can argue as in Theorem \ref{thm:dobule series =00005Cchi},
and observe, by Stirling formula (\ref{eq:stirling}), that
\[
\sum_{\rho:\left|\gamma\right|\leq U}\frac{1}{\left|\zeta^{\prime}(\rho)\right|}\sum_{\rho_{\chi}:\gamma_{\chi}=0}\left|\frac{\Gamma\left(\rho\right)\Gamma\left(\rho_{\chi}\right)}{\Gamma\left(\rho_{\chi}+\rho+1\right)}\right|\ll N\left(4,\chi\right)\sum_{\rho:\left|\gamma\right|\leq U}\frac{\left|\Gamma\left(\rho\right)\right|}{\left|\zeta^{\prime}(\rho)\right|\left|\Gamma\left(\rho+\frac{3}{2}\right)\right|}
\]
\[
\ll\log\left(q\right)\sum_{\rho}\frac{1}{\left|\zeta^{\prime}(\rho)\right|\left|\gamma\right|^{3/2}}
\]
which is convergent, by Conjecture \ref{conj:simple zeros}.

Finally, we obtain
\[
\frac{N\log\left(N\right)}{\varphi\left(q\right)}\sum_{\rho:\left|\gamma\right|\leq U}\frac{1}{\left|\zeta^{\prime}(\rho)\right|}\sum_{\chi\neq\chi_{0}}\sum_{\rho_{\chi}:\left|\gamma_{\chi}\right|\leq T}\left|\frac{\Gamma\left(\rho\right)\Gamma\left(\rho_{\chi}\right)}{\Gamma\left(\rho_{\chi}+\rho+1\right)}\right|
\]
\[
\ll N\log\left(N\right)\log\left(q\right)\left(TU\right)^{1/4}\log\left(qT\right)^{2}\log\left(U\right)^{1/2}
\]
and since $U\asymp T\asymp N$ we can conclude
\[
\frac{N\log\left(N\right)}{\varphi\left(q\right)}\sum_{\rho:\left|\gamma\right|\leq U}\frac{1}{\left|\zeta^{\prime}(\rho)\right|}\sum_{\chi\neq\chi_{0}}\sum_{\rho_{\chi}:\left|\gamma_{\chi}\right|\leq T}\left|\frac{\Gamma\left(\rho\right)\Gamma\left(\rho_{\chi}\right)}{\Gamma\left(\rho_{\chi}+\rho+1\right)}\right|
\]
\[
\ll N^{3/2}\log\left(N\right)^{9/4}\log\left(q\right)^{2}.
\]
Again, combining all the pieces, we get, taking $\theta=1/2-2\varepsilon$,
that
\[
\sum_{1<q\leq N^{\theta}}\frac{1}{\varphi\left(q\right)}\left|\sum_{\chi\neq\chi_{0}}\overline{\chi}\left(N\right)\sum_{n\leq N\beta}\sum_{m\leq N\beta-n}\Lambda\left(n\right)\chi\left(n\right)\mu\left(m\right)f\left(\frac{n+m}{N}\right)\right|
\]
\[
\ll_{f,\beta}N^{3/2}\log\left(N\right)^{9/4}\sum_{1<q\leq N^{1/2-2\varepsilon}}\log\left(q\right)^{2}+N^{5/2-\delta}\log\left(N\right)^{5}\sum_{1<q\leq N^{1/2-2\varepsilon}}\log\left(q\right)^{5}
\]
\[
\ll_{f,\beta,\varepsilon}N^{3/2}\log\left(N\right)^{5}\sum_{q\leq N^{1/2-2\varepsilon}}\log\left(q\right)^{5}\ll_{f,\beta,\varepsilon}N^{2-\varepsilon}
\]
and the claim follows.
\end{proof}

\section{weighted averages of diagonal twisted eh problem with $g(n)=\mu(n)\log(n)$ }

In this section we will use the previous results to show that, also
in this case, we are able to show that an analogous weighted average
for the diagonal twisted EH problem with $g(n)=\mu(n)\log(n)$, that
is the average
\[
\sum_{N\alpha<n\leq N\beta}\sum_{m\leq N\beta-n}\Lambda\left(n\right)\chi\left(n\right)\mu\left(m\right)\log\left(m\right)f\left(\frac{n+m}{N}\right)
\]
is consistent with Conjecture \ref{conj:DiagonalEH}. The proofs are
similar, so we omit some repetitive details and we focus on the main
difference between these two cases. We start again considering the
weight in the Sobolev space $W^{2,1}$. 
\begin{thm}
\label{thm:explicit formula M tilde}Assume the hypotheses of Theorem
\ref{thm:Main1}. Then
\[
\frac{1}{\varphi\left(q\right)}\sum_{\chi\neq\chi_{0}}\overline{\chi}\left(N\right)\sum_{n\leq N\beta}\sum_{m\leq N\beta-n}\Lambda\left(n\right)\chi\left(n\right)\mu\left(m\right)\log\left(m\right)f\left(\frac{n+m}{N}\right)
\]
\[
=-\frac{1}{\varphi\left(q\right)}\sum_{\rho:\left|\gamma\right|\leq U}\frac{N^{\rho}\Gamma\left(\rho\right)}{\zeta^{\prime}(\rho)}\sum_{\chi\neq\chi_{0}}\overline{\chi}\left(N\right)\sum_{\rho_{\chi}:\left|\gamma_{\chi}\right|\leq T}\frac{N^{\rho_{\chi}}\Gamma\left(\rho_{\chi}\right)}{\Gamma\left(\rho_{\chi}+\rho+2\right)}\int_{0}^{\beta}f^{\prime\prime}\left(w\right)\mathcal{H}\left(N,w,\rho,\rho_{\chi}\right)w^{\rho+\rho_{\chi}+1}dw
\]
\[
+O_{\varepsilon}\left(\frac{1}{N}\int_{0}^{\beta}\left|f^{\prime\prime}\left(w\right)\right|\left[\mathcal{E}_{1}^{*}\left(N,w,q,U,T\right)+\mathcal{E}_{2}^{*}\left(N,w,q,U,T\right)\right]dw\right)
\]
where
\begin{equation}
\mathcal{H}\left(N,w,\rho,\rho_{\chi}\right):=\log\left(Nw\right)+\psi^{\left(0\right)}\left(\rho+1\right)-\psi^{\left(0\right)}\left(\rho_{\chi}+\rho+2\right)-\frac{1}{\rho}\label{eq:H storto}
\end{equation}
and $\psi^{\left(0\right)}\left(z\right)$ is the Digamma function,

\begin{equation}
\mathcal{E}_{1}^{*}\left(N,w,q,U,T\right):=\varphi\left(q\right)N^{3/2+\varepsilon}\log\left(q\right)^{3}\log\left(N\right)^{3}\left(1+w^{3/2+\varepsilon}\right)\left(1+\left|\log\left(w\right)\right|^{3}\right)\label{eq:E_1 star}
\end{equation}
\[
\times\left[1+\frac{N\left(1+w\right)\log\left(qT\right)^{2}}{T}+\frac{N^{1-\varepsilon}\left(w^{1-\varepsilon}+1\right)}{U^{1-\varepsilon}}\right],
\]
\[
\mathcal{E}_{2}^{*}\left(N,w,q,U,T\right)=\varphi(q)N^{1+\varepsilon}\log(q)\log\left(N\right)^{3}\left(1+\left|\log\left(w\right)\right|^{3}\right)
\]
\begin{equation}
\times\left(1+w^{1+\varepsilon}\right)\left[1+\frac{N\log\left(qT\right)^{2}}{T}\right]\left[1+\frac{N^{1-1\varepsilon}}{U^{1-1\varepsilon}}\right].\label{eq:E_2 star}
\end{equation}
and the implicit constant depends only on $\varepsilon$. 
\end{thm}

\begin{proof}
From Theorem \ref{thm:MAINCGZ} we have
\[
\frac{1}{\varphi\left(q\right)}\sum_{\chi\neq\chi_{0}}\overline{\chi}\left(N\right)\sum_{n\leq N\beta}\sum_{m\leq N\beta-n}\Lambda\left(n\right)\chi\left(n\right)\mu\left(m\right)\log\left(m\right)f\left(\frac{n+m}{N}\right)
\]
\begin{equation}
=\frac{1}{N\varphi\left(q\right)}\sum_{\chi\neq\chi_{0}}\overline{\chi}\left(N\right)\int_{0}^{\beta}f^{\prime\prime}\left(w\right)\int_{0}^{Nw}\widetilde{M}\left(s\right)\psi\left(Nw-s,\chi\right)ds.\label{eq:main applied to M tilde}
\end{equation}
We made a dissection like (\ref{eq:dissect1}), (\ref{eq:dissect2})
and (\ref{eq:dissect3}). Recalling the definition of Digamma function
\[
\psi^{\left(0\right)}\left(z\right):=\frac{\Gamma^{\prime}\left(z\right)}{\Gamma\left(z\right)},z\neq0,-1,-2,\dots
\]
(see \cite{O2010}, $5.2.2$) and the representation
\[
\psi^{\left(0\right)}\left(a\right)-\psi^{\left(0\right)}\left(a+b\right)=\frac{1}{B\left(a,b\right)}\int_{0}^{1}\log\left(s\right)s^{a-1}\left(1-s\right)^{b-1}ds,\,\text{Re}(a)>0,\,\text{Re}(b)>0
\]
which is a simple consequence of the differentiation under the integral
sign, we get that the main term of (\ref{eq:main applied to M tilde})
is
\[
-\frac{1}{\varphi\left(q\right)}\sum_{\rho:\left|\gamma\right|\leq U}\frac{N^{\rho}\Gamma\left(\rho\right)}{\zeta^{\prime}(\rho)}\sum_{\chi\neq\chi_{0}}\overline{\chi}\left(N\right)\sum_{\rho_{\chi}:\left|\gamma_{\chi}\right|\leq T}\frac{N^{\rho_{\chi}}\Gamma\left(\rho_{\chi}\right)}{\Gamma\left(\rho_{\chi}+\rho+2\right)}\int_{0}^{\beta}f^{\prime\prime}\left(w\right)\mathcal{H}\left(N,w,\rho,\rho_{\chi}\right)w^{\rho+\rho_{\chi}+1}dw
\]
where $\mathcal{H}\left(N,w,\rho,\rho_{\chi}\right)$ is (\ref{eq:H storto}).
So it remains to evaluate the error term. Now, by (\ref{eq:M tilde explicit})
and arguing as in Theorem \ref{thm:Main1}, we may deduce that the
uniform bound for $0<x\leq CN,\,C>0$
\begin{equation}
\widetilde{M}\left(x\right)\ll\log\left(N\right)N^{1/2+\varepsilon}\label{eq:Unif M tilde}
\end{equation}
hence defining
\[
I_{1}^{*}=I_{1}^{*}(w,N,q,T):=\int_{0}^{Nw}\widetilde{M}\left(s\right)\sum_{\chi\neq\chi_{0}}\overline{\chi}\left(N\right)R_{1,\chi}(Nw-s,T,q)ds
\]
and by (\ref{eq:Unif M tilde}) and following the same calculations
of $I_{1}$ in Theorem \ref{thm:Main1}, we get
\begin{equation}
I_{1}^{*}\ll\varphi(q)\log\left(q\right)N^{3/2+\varepsilon}\log\left(N\right)^{3}w^{3/2+\varepsilon}\left(1+\left|\log\left(w\right)\right|^{3}\right)\left[1+\frac{N\left(1+w\right)\log\left(qT\right)^{2}}{T}\right].\label{eq:I_1 star first bound}
\end{equation}
Then we consider
\[
I_{2}^{*}=I_{2}^{*}(w,N,q,U):=\int_{0}^{Nw}\widetilde{R}_{2}\left(s,U\right)\sum_{\chi\neq\chi_{0}}\overline{\chi}\left(N\right)\psi(Nw-s,\chi)ds
\]
where $\widetilde{R}\left(s,U\right)$ is the error term in (\ref{eq:M tilde explicit}),
that is
\[
\widetilde{M}\left(s\right)=\sum_{\rho:\left|\gamma\right|\leq U}\frac{x^{\rho}}{\zeta^{\prime}(\rho)\rho}\left(\log\left(x\right)-\frac{1}{\rho}\right)+\widetilde{R}_{2}\left(s,U\right).
\]
Note that, compared to the estimate (\ref{eq:I_2 bound}) and (\ref{eq:I_2 bound-1})
we have only the additional term
\[
\varphi\left(q\right)N^{1/2}w^{1/2}\left|\log\left(Nw\right)\right|\left|\log\left(qNw\right)\right|\int_{0}^{yw}\left|\log\left(s\right)\right|ds
\]
\[
\ll\varphi\left(q\right)N^{3/2}w^{3/2}\log\left(Nw\right)^{2}\left|\log\left(qNw\right)\right|
\]
and this produces the bound
\[
I_{2}^{*}\ll\varphi\left(q\right)N^{3/2+\varepsilon}\log\left(q\right)^{3}\log\left(N\right)^{3}
\]
\begin{equation}
\times\left[w^{3/2}\left(1+w^{\varepsilon}\right)\left(1+\left|\log\left(w\right)\right|^{3}\right)+\frac{N^{1-\varepsilon}\left(1+\left|\log\left(w\right)\right|^{3}\right)w^{1/2+\varepsilon}\left(w^{2-\varepsilon}+1\right)}{U^{1-\varepsilon}}\right].\label{eq:I_2 star second bound}
\end{equation}
Finally, we evaluate
\[
I_{3}^{*}=I_{3}^{*}(w,N,q,U):=\int_{0}^{Nw}\widetilde{R}_{2}\left(s,U\right)\sum_{\chi\neq\chi_{0}}\overline{\chi}\left(N\right)R_{1,\chi}(Nw-s,T,q)ds.
\]
Arguing as in $I_{3}$ of Theorem \ref{thm:Main1}, we observe that
the only part we need to modified is 
\[
\int_{0}^{Nw}\widetilde{R_{2}}(s,U)^{2}ds\ll\int_{0}^{Nw}\left[1+s^{2\varepsilon}+\log\left(s\right)^{2}+\left[\frac{\left(s^{2}+s^{2\varepsilon}\right)\left(\left|\log\left(s\right)\right|^{2}+1\right)}{U^{2-2\varepsilon}}\right]\right]ds
\]
\[
\ll N^{1+2\varepsilon}\log\left(N\right)^{2}\left(1+\log\left(w\right)^{2}\right)\left(1+w^{1+2\varepsilon}\right)\left[1+\frac{N^{2-2\varepsilon}\left(1+w^{2-2\varepsilon}\right)\left(1+\log\left(w\right)^{2}\right)}{U^{2-2\varepsilon}}\right].
\]
 So 
\begin{equation}
I_{3}^{*}\ll\varphi(q)N^{1+\varepsilon}\log(q)\log\left(N\right)^{3}\left(1+\left|\log\left(w\right)\right|^{3}\right)\left(1+w^{1+\varepsilon}\right)\label{eq:I_3 star 1}
\end{equation}
\begin{equation}
\times\left[1+\frac{N\log\left(qT\right)^{2}}{T}\right]\left[1+\frac{N^{1-1\varepsilon}}{U^{1-1\varepsilon}}\right].\label{eq:I_3 star 2}
\end{equation}
Again, for simplicity, we observe that we can use the common bound
\begin{equation}
I_{1}^{*},I_{2}^{*}\ll\varphi\left(q\right)N^{3/2+\varepsilon}\log\left(q\right)^{3}\log\left(N\right)^{3}\left(1+w^{3/2+\varepsilon}\right)\left(1+\left|\log\left(w\right)\right|^{3}\right)\label{eq:I_1 star I_2 star 1}
\end{equation}
\begin{equation}
\times\left[1+\frac{N\left(1+w\right)\log\left(qT\right)^{2}}{T}+\frac{N^{1-\varepsilon}\left(w^{1-\varepsilon}+1\right)}{U^{1-\varepsilon}}\right]\label{eq:I_1 star I_2 star 2}
\end{equation}
and the result follows.
\end{proof}
So now we are able to prove the bound for the diagonal EH with $g(n)=\mu(n)\log(n)$
in the case that the weight is in Sobolev $W^{2,1}.$
\begin{thm}
\label{thm:dEH bound with log}Assume all the hypotheses of Theorem
\ref{thm:Main1} and fix $0<\varepsilon<1/2$ and $\theta=1-2\varepsilon$.
Then, the following weighted average of diagonal twisted EH holds
\[
\sum_{\underset{{\scriptstyle \left(N,q\right)=1}}{1<q\leq N^{1-2\varepsilon}}}\frac{1}{\varphi\left(q\right)}\left|\sum_{\chi\neq\chi_{0}}\overline{\chi}\left(N\right)\sum_{n\leq N\beta}\sum_{m\leq N\beta-n}\Lambda\left(n\right)\chi\left(n\right)\mu\left(m\right)\log\left(m\right)f\left(\frac{n+m}{N}\right)\right|
\]
\[
\ll_{\varepsilon}N^{2-\varepsilon}E_{1}\left(f^{\prime\prime}\right)
\]
where
\[
E_{1}\left(f^{\prime\prime}\right):=\int_{0}^{\beta}\left|f^{\prime\prime}\left(w\right)\right|\left(1+\left|\log\left(w\right)\right|^{3}\right)\left(1+w^{5/2+\varepsilon}\right)dw.
\]
and the implicit constant depends only on $\varepsilon$.
\end{thm}

\begin{proof}
From Theorem \ref{thm:explicit formula M tilde} we know that we need
to bound the explicit formula. We start from the main term. Recalling
the asymptotic formula
\begin{equation}
\psi^{\left(0\right)}\left(z\right)\sim\log\left(z\right)-\frac{1}{2z}-\sum_{k\geq0}\frac{B_{2k}}{2kz^{2k}}\label{eq:asymp digamma}
\end{equation}
as $\left|z\right|\rightarrow+\infty,\,\left|\arg\left(z\right)\right|\leq\pi-\epsilon,\,\epsilon>0$
(see \cite{O2010}, $5.11.2$) where $\log\left(z\right)$ is taken
with its principal branch, we obtain
\[
\psi^{\left(0\right)}\left(\rho+1\right)-\psi^{\left(0\right)}\left(\rho_{\chi}+\rho+2\right)-\frac{1}{\rho}
\]
\[
\ll\log\left(\left|\gamma\right|\right)+\log\left(1+\left|\gamma+\gamma_{\chi}\right|\right)+\frac{1}{\left|\rho\right|}
\]
\[
\ll\log\left(T+U\right)
\]
hence
\[
\frac{1}{\varphi(q)}\sum_{\rho:\left|\gamma\right|\leq U}\frac{N^{\rho}\Gamma\left(\rho\right)}{\zeta^{\prime}(\rho)}\sum_{\chi\neq\chi_{0}}\overline{\chi}\left(N\right)\sum_{\rho_{\chi}:\left|\gamma_{\chi}\right|\leq T}\frac{N^{\rho_{\chi}}\Gamma\left(\rho_{\chi}\right)}{\Gamma\left(\rho_{\chi}+\rho+2\right)}\int_{0}^{\beta}f^{\prime\prime}\left(w\right)\mathcal{H}\left(N,w,\rho,\rho_{\chi}\right)w^{\rho+\rho_{\chi}+1}dw
\]
\[
\ll\frac{N\log\left(N\right)\log\left(T+U\right)}{\varphi(q)}\sum_{\rho:\left|\gamma\right|\leq U}\left|\frac{\Gamma\left(\rho\right)}{\zeta^{\prime}(\rho)}\right|\sum_{\chi\neq\chi_{0}}\sum_{\rho_{\chi}:\left|\gamma_{\chi}\right|\leq T}\left|\frac{\Gamma\left(\rho_{\chi}\right)}{\Gamma\left(\rho_{\chi}+\rho+2\right)}\right|
\]
\[
\times\int_{0}^{\beta}\left|f^{\prime\prime}\left(w\right)\right|\left(1+\left|\log\left(w\right)\right|\right)w^{2}dw
\]
hence if we take $T\asymp U\asymp N$ we get the bound
\[
N\log\left(N\right)^{2}\max_{\chi\mod q}\sum_{\rho:\left|\gamma\right|\leq U}\left|\frac{\Gamma\left(\rho\right)}{\zeta^{\prime}(\rho)}\right|\sum_{\rho_{\chi}:\left|\gamma_{\chi}\right|\leq T}\left|\frac{\Gamma\left(\rho_{\chi}\right)}{\Gamma\left(\rho_{\chi}+\rho+2\right)}\right|
\]
\[
\times\int_{0}^{\beta}\left|f^{\prime\prime}\left(w\right)\right|\left(1+\left|\log\left(w\right)\right|\right)w^{2}dw
\]
and arguing as in (\ref{eq:esti chi mod q}) we get that the double
series is absolutely convergent and bounded by $\log\left(q\right)$.
Hence the bound for the main term is
\[
N\log\left(N\right)^{2}\log\left(q\right)\int_{0}^{\beta}\left|f^{\prime\prime}\left(w\right)\right|\left(1+\left|\log\left(w\right)\right|\right)w^{2}dw.
\]
Now let us focus on the error term. By (\ref{eq:E_1 star}), (\ref{eq:E_2 star}),
from the choice $T\asymp U\asymp N$ and trivial manipulations we
get
\[
\mathcal{E}_{1}^{*}\left(N,w,q,U,T\right)\ll\varphi\left(q\right)N^{3/2+\varepsilon}\log\left(q\right)^{5}\log\left(N\right)^{5}\left(1+w^{5/2+\varepsilon}\right)\left(1+\left|\log\left(w\right)\right|^{3}\right)
\]
\[
\mathcal{E}_{2}^{*}\left(N,w,q,U,T\right)\ll\varphi\left(q\right)N^{1+\varepsilon}\log(q)^{2}\log\left(N\right)^{5}\left(1+\left|\log\left(w\right)\right|^{3}\right)\left(1+w^{1+\varepsilon}\right)
\]
hence, putting together all the pieces and taking the sum $\sum_{\underset{{\scriptstyle \left(N,q\right)=1}}{q\leq N^{\theta}}}$
we get
\[
\sum_{\underset{{\scriptstyle \left(N,q\right)=1}}{q\leq N^{\theta}}}\frac{1}{\varphi\left(q\right)}\left|\sum_{\chi\neq\chi_{0}}\overline{\chi}\left(N\right)\sum_{n\leq N\beta}\sum_{m\leq N\beta-n}\Lambda\left(n\right)\chi\left(n\right)\mu\left(m\right)\log\left(m\right)f\left(\frac{n+m}{N}\right)\right|
\]
\[
\ll_{\varepsilon}N^{2-\varepsilon}E_{1}\left(f^{\prime\prime}\right)
\]
with $\theta=1-2\varepsilon$ and 
\[
E_{1}\left(f^{\prime\prime}\right):=\int_{0}^{\beta}\left|f^{\prime\prime}\left(w\right)\right|\left(1+\left|\log\left(w\right)\right|^{3}\right)\left(1+w^{5/2+\varepsilon}\right)dw.
\]
\end{proof}
Now we consider the case of $f\in\mathcal{C}^{\delta}\left(\mathbb{R}\right)$.
Also in this case, we start with an explicit formula for the average. 
\begin{thm}
\label{thm:dEh with log explicit formula Zyg}Assume all the hypotheses
of Theorem \ref{thm:explicit discrete}. Then
\[
\frac{1}{\varphi\left(q\right)}\sum_{\chi\neq\chi_{0}}\overline{\chi}\left(N\right)\sum_{n\leq N\beta}\sum_{m\leq N\beta-n}\Lambda\left(n\right)\chi\left(n\right)\mu\left(m\right)\log\left(m\right)f\left(\frac{n+m}{N}\right)
\]
\[
=-\frac{1}{\varphi\left(q\right)}\sum_{\rho:\left|\gamma\right|\leq U}\frac{\Gamma\left(\rho\right)}{\zeta^{\prime}(\rho)}\sum_{\chi\neq\chi_{0}}\overline{\chi}\left(N\right)\sum_{\rho_{\chi}:\left|\gamma_{\chi}\right|\leq T}\frac{\Gamma\left(\rho_{\chi}\right)}{\Gamma\left(\rho_{\chi}+\rho+2\right)}
\]
\[
\times\sum_{k\leq\beta N-1}\Delta_{1/N}^{2}\left(f,\frac{k}{N}\right)\mathcal{H}\left(N,\frac{k}{N},\rho,\rho_{\chi}\right)k^{\rho+\rho_{\chi}+1}
\]
\[
+O_{f,\beta,\varepsilon}\left(N^{5/2-\delta}\log\left(N\right)\log\left(qT\right)^{2}+\frac{N^{3-\delta}\log\left(N\right)\log\left(qNT\right)^{2}}{T}\right..
\]
\[
\left.+\log\left(q\right)^{3}\varphi\left(q\right)N^{5/2-\delta+\varepsilon}\log\left(N\right)^{3}\left[1+\frac{N\log\left(qT\right)^{2}}{T}\right]\left[1+\frac{N^{1-\varepsilon}}{U^{1-\varepsilon}}\right]\right).
\]
\end{thm}

\begin{proof}
Arguing as in Theorem \ref{thm:explicit discrete} we have to deal
with
\[
\frac{1}{\varphi\left(q\right)}\sum_{\chi\neq\chi_{0}}\overline{\chi}\left(N\right)\sum_{n\leq N\beta}\sum_{m\leq N\beta-n}\Lambda\left(n\right)\chi\left(n\right)\mu\left(m\right)\log\left(m\right)f\left(\frac{n+m}{N}\right)
\]
\[
=\frac{1}{\varphi\left(q\right)}\sum_{k\leq N\beta-1}\Delta_{1/N}^{2}\left(f,\frac{k}{N}\right)\sum_{\chi\neq\chi_{0}}\overline{\chi}\left(N\right)\int_{0}^{k}\psi\left(s,\chi\right)\widetilde{M}\left(k-s\right)ds
\]
\[
+O_{f,\beta}\left(\frac{N^{-\delta}}{\varphi(q)}\sum_{k\leq N\beta-1}\sum_{\chi\neq\chi_{0}}\left|\sum_{u=0}^{k}\psi(\chi,u)\mu(k-u)\log\left(k-u\right)\right|\right).
\]
By the explicit formula in Theorem \ref{thm:explicit formula M tilde}
we have
\[
\frac{1}{\varphi\left(q\right)}\sum_{k\leq N\beta-1}\Delta_{1/N}^{2}\left(f,\frac{k}{N}\right)\sum_{\chi\neq\chi_{0}}\overline{\chi}\left(N\right)\int_{0}^{k}\psi\left(s,\chi\right)\widetilde{M}\left(k-s\right)ds
\]
\[
=\frac{1}{\varphi\left(q\right)}\sum_{\rho:\left|\gamma\right|\leq U}\frac{\Gamma\left(\rho\right)}{\zeta^{\prime}(\rho)}\sum_{\chi\neq\chi_{0}}\overline{\chi}\left(N\right)\sum_{\rho_{\chi}:\left|\gamma_{\chi}\right|\leq T}\frac{\Gamma\left(\rho_{\chi}\right)}{\Gamma\left(\rho_{\chi}+\rho+2\right)}
\]
\[
\times\sum_{k\leq\beta N-1}\Delta_{1/N}^{2}\left(f,\frac{k}{N}\right)\mathcal{H}\left(N,\frac{k}{N},\rho,\rho_{\chi}\right)k^{\rho+\rho_{\chi}+1}
\]
\[
+O_{f}\left(\frac{N^{-\delta}}{\varphi\left(q\right)}\sum_{k\leq N\beta-1}\left[I_{1}^{*}\left(k/N,N,q,T\right)+I_{2}^{*}\left(k/N,N,q,U\right)+I_{3}^{*}\left(k/N,N,q,T,U\right)\right]\right)
\]
\[
+O_{f,\beta}\left(\frac{N^{-\delta}}{\varphi(q)}\sum_{k\leq N\beta-1}\sum_{\chi\neq\chi_{0}}\left|\sum_{u=0}^{k}\psi\left(u,\chi\right)\mu\left(k-u\right)\log\left(k-u\right)\right|\right).
\]
We start to analyze the first error term. From the same calculations
of (\ref{eq:I_1 star I_2 star 1}), (\ref{eq:I_1 star I_2 star 2}),
and recalling that the convolution now is from $0$ to $k$ instead
of $0$ to $Nw$, we obtain the bound
\[
\log\left(q\right)^{3}\varphi\left(q\right)N^{-\delta}\sum_{k\leq N\beta-1}k^{3/2+\varepsilon}\log\left(k\right)^{3}\left[1+\frac{k\log\left(qT\right)^{2}}{T}+\frac{k^{1-\varepsilon}}{U^{1-\varepsilon}}\right]
\]
\[
\ll_{\varepsilon,\beta}\log\left(q\right)^{3}\varphi\left(q\right)N^{5/2-\delta+\varepsilon}\log\left(N\right)^{3}\left[1+\frac{N\log\left(qT\right)^{2}}{T}+\frac{N^{1-\varepsilon}}{U^{1-\varepsilon}}\right]
\]
and in the same spirit, considering (\ref{eq:I_3 star 1}) and (\ref{eq:I_3 star 2}),
we obtain the bound
\[
\log(q)\varphi\left(q\right)N^{-\delta}\sum_{k\leq N\beta-1}k^{1+\varepsilon}\log\left(k\right)^{3}\left[1+\frac{k\log\left(qT\right)^{2}}{T}\right]\left[1+\frac{k^{1-\varepsilon}}{U^{1-1\varepsilon}}\right]
\]
\[
\ll_{\varepsilon,\beta}\log(q)\varphi\left(q\right)N^{2+\varepsilon-\delta}\log\left(N\right)^{3}\left[1+\frac{N\log\left(qT\right)^{2}}{T}\right]\left[1+\frac{N^{1-\varepsilon}}{U^{1-1\varepsilon}}\right].
\]
Now, we consider the second error term. Again by the explicit formula
in Theorem \ref{thm:Explicit psi} we get
\[
\frac{N^{-\delta}}{\varphi(q)}\sum_{k\leq N\beta-1}\sum_{\chi\neq\chi_{0}}\left|\sum_{u=0}^{k}\psi\left(u,\chi\right)\mu\left(k-u\right)\log\left(k-u\right)\right|
\]
\[
=\frac{N^{-\delta}}{\varphi(q)}\sum_{k\leq N\beta-1}\sum_{\chi\neq\chi_{0}}\left|\sum_{\rho_{\chi}:\left|\gamma_{\chi}\right|\leq T}\frac{1}{\rho_{\chi}}\sum_{u=1}^{k-1}u^{\rho_{\chi}}\mu\left(k-u\right)\log\left(k-u\right)\right|
\]
\[
+O_{\beta}\left(N^{2-\delta}\log\left(N\right)^{2}+N^{-\delta}\log\left(N\right)\sum_{k\leq N\beta-1}\sum_{u=1}^{k-1}\frac{u\log\left(quT\right)^{2}}{T}\right).
\]
Again we have
\[
N^{-\delta}\log\left(N\right)\sum_{k\leq N\beta-1}\sum_{u=1}^{k-1}\frac{u\log\left(quT\right)^{2}}{T}\ll\frac{N^{3-\delta}\log\left(N\right)\log\left(qNT\right)^{2}}{T}
\]
and
\[
\frac{N^{-\delta}}{\varphi(q)}\sum_{k\leq N\beta-1}\sum_{\chi\neq\chi_{0}}\left|\sum_{\rho_{\chi}:\left|\gamma_{\chi}\right|\leq T}\frac{1}{\rho_{\chi}}\sum_{u=1}^{k-1}u^{\rho_{\chi}}\mu(k-u)\log\left(k-u\right)\right|
\]
\[
\ll\frac{N^{-\delta}\log\left(N\right)}{\varphi(q)}\sum_{k\leq N\beta-1}k^{3/2}\sum_{\chi\neq\chi_{0}}\sum_{\rho_{\chi}:\left|\gamma_{\chi}\right|\leq T}\frac{1}{\left|\rho_{\chi}\right|}\ll N^{5/2-\delta}\log\left(N\right)\log\left(qT\right)^{2}
\]
and the thesis follows combining all the pieces.
\end{proof}
Finally, we can prove the bound for the diagonal EH problem with $f\in\mathcal{C}^{\delta}\left(\mathbb{R}\right)$.
\begin{thm}
\label{thm:dEH average with log}Assume all the hypotheses of Theorem
\ref{thm:explicit discrete} and fix $\theta=\delta-1-2\varepsilon$
if $\delta\in\left[\frac{3}{2},2\right)$ with $0<\varepsilon<\left(\delta-1\right)/2$
and $\theta=\frac{1}{2}-2\varepsilon$ if $\delta\in\left[1,3/2\right)$
with $0<\varepsilon<1/4$. Then
\[
\sum_{\underset{{\scriptstyle \left(N,q\right)=1}}{1<q\leq N^{\theta}}}\frac{1}{\varphi\left(q\right)}\left|\sum_{\chi\neq\chi_{0}}\overline{\chi}\left(N\right)\sum_{n\leq N\beta}\sum_{m\leq N\beta-n}\Lambda\left(n\right)\chi\left(n\right)\mu\left(m\right)\log\left(m\right)f\left(\frac{n+m}{N}\right)\right|
\]
\[
\ll_{f,\beta,\varepsilon,\delta}N^{2-\varepsilon}.
\]
\end{thm}

\begin{proof}
We start again considering the case $\delta\in\left[\frac{3}{2},2\right)$.
By the explicit formula in Theorem \ref{thm:dEh with log explicit formula Zyg}
and (\ref{eq:asymp digamma}) we get
\[
\frac{1}{\varphi\left(q\right)}\sum_{\rho:\left|\gamma\right|\leq U}\frac{\Gamma\left(\rho\right)}{\zeta^{\prime}(\rho)}\sum_{\chi\neq\chi_{0}}\overline{\chi}\left(N\right)\sum_{\rho_{\chi}:\left|\gamma_{\chi}\right|\leq T}\frac{\Gamma\left(\rho_{\chi}\right)}{\Gamma\left(\rho_{\chi}+\rho+2\right)}
\]
\[
\times\sum_{k\leq N\beta-1}\Delta_{1/N}^{2}\left(f,\frac{k}{N}\right)\mathcal{H}\left(N,\frac{k}{N},\rho,\rho_{\chi}\right)k^{\rho+\rho_{\chi}+1}
\]
\[
\ll_{f,\beta}\frac{N^{3-\delta}\log\left(T+U\right)}{\varphi(q)}\sum_{\rho:\left|\gamma\right|\leq U}\left|\frac{\Gamma\left(\rho\right)}{\zeta^{\prime}(\rho)}\right|\sum_{\chi\neq\chi_{0}}\sum_{\rho_{\chi}:\left|\gamma_{\chi}\right|\leq T}\left|\frac{\Gamma\left(\rho_{\chi}\right)}{\Gamma\left(\rho_{\chi}+\rho+2\right)}\right|
\]
then, taking $T\asymp U\asymp N$ and arguing as in Theorem \ref{thm:dEHtwis bound}
for the analysis of the double series, we obtain the bound
\[
\ll_{f,\beta,\delta}\frac{N^{3-\delta}\log\left(N\right)\log\left(q\right)}{\varphi(q)}\ll_{f,\beta,\delta}\frac{N^{3-\delta}\log\left(N\right)^{2}}{\varphi(q)}.
\]
For what concerning the error term, with the choice $T\asymp U\asymp N$,
we obtain 
\[
E_{\max}^{**}:=N^{5/2-\delta+\varepsilon}\log\left(N\right)\log\left(qT\right)^{2}+\frac{N^{3-\delta}\log\left(N\right)\log\left(qNT\right)^{2}}{T}
\]
\[
+\log\left(q\right)^{3}N^{5/2-\delta+\varepsilon}\log\left(N\right)^{3}\left[1+\frac{N\log\left(qT\right)^{2}}{T}\right]\left[1+\frac{N^{1-\varepsilon}}{U^{1-1\varepsilon}}\right]
\]
\[
\ll_{\varepsilon,\theta}N^{5/2-\delta+\varepsilon}\log\left(N\right)^{8}
\]
hence, taking $\theta=\delta-1-2\varepsilon$, we obtain
\[
\sum_{1<q\leq N^{\delta-1-2\varepsilon}}\frac{1}{\varphi\left(q\right)}\left|\sum_{\chi\neq\chi_{0}}\overline{\chi}\left(N\right)\sum_{n\leq N\beta}\sum_{m\leq N\beta-n}\Lambda\left(n\right)\chi\left(n\right)\mu\left(m\right)\log\left(m\right)f\left(\frac{n+m}{N}\right)\right|
\]
\[
\ll_{f,\beta,\varepsilon,\delta}N^{2-2\varepsilon}\log\left(N\right)^{2}+N^{3/2-2\varepsilon}\log\left(N\right)^{8}\ll_{f,\beta,\varepsilon,\delta}N^{2-\varepsilon}.
\]
Now we focus on the case $\delta\in\left[1,\frac{3}{2}\right)$. Assuming
again that $U\asymp T\asymp N$ the bound of the error term is the
same of the previous part, so
\[
E_{\max}^{**}\ll_{\varepsilon}N^{5/2-\delta+\varepsilon}\log\left(N\right)^{8}
\]
So now we focus on the main term. Clearly
\[
\mathcal{H}\left(N,\frac{k}{N},\rho,\rho_{\chi}\right)=\log\left(k\right)+\psi^{\left(0\right)}\left(\rho+1\right)-\psi^{\left(0\right)}\left(\rho_{\chi}+\rho+2\right)-\frac{1}{\rho}
\]
hence the significant difference here respect the previous results
is the factor $\log\left(k\right)$ in the sum. Using summation by
parts and property (\ref{eq:prod fow diff}) we observe that 
\[
\sum_{k\leq N\beta-1}\Delta_{1/N}^{2}\left(f,\frac{k}{N}\right)\log\left(k\right)k^{\rho+\rho_{\chi}+1}=\Delta_{1/N}\left(f,\beta\right)\log\left(N\beta\right)\left(N\beta-1\right)^{\rho+\rho_{\chi}+1}
\]
\[
-\sum_{k\leq N\beta-1}\Delta_{1/N}\left(f,\frac{k}{N}\right)\log\left(k\right)\left(k^{\rho+\rho_{\chi}+1}-\left(k-1\right)^{\rho+\rho_{\chi}+1}\right).
\]
From (\ref{eq:LogLip}) we have
\[
\Delta_{1/N}\left(f,\beta\right)\log\left(N\beta\right)\left(N\beta-1\right)^{\rho+\rho_{\chi}+1}\ll_{\beta}N\log\left(N\right)^{2}
\]
and 
\[
\sum_{k\leq N\beta-1}\Delta_{1/N}\left(f,\frac{k}{N}\right)\log\left(k\right)\left(k^{\rho+\rho_{\chi}+1}-\left(k-1\right)^{\rho+\rho_{\chi}+1}\right)
\]
\[
\ll_{f,\beta}N\log\left(N\right)^{2}\left|\rho+\rho_{\chi}+1\right|
\]
using the same ideas of Theorem \ref{thm:MainDiscrete}. So, since
the part $\psi^{\left(0\right)}\left(\rho+1\right)-\psi^{\left(0\right)}\left(\rho_{\chi}+\rho+2\right)-\frac{1}{\rho}$
produces a bound $\log\left(T+U\right)\ll\log\left(N\right)$, for
bounding the main term it is just enough to consider
\[
\frac{N\log\left(N\right)^{2}}{\varphi\left(q\right)}\sum_{\rho:\left|\gamma\right|\leq N}\frac{1}{\left|\zeta^{\prime}(\rho)\right|}\sum_{\chi\neq\chi_{0}}\sum_{\rho_{\chi}:\left|\gamma_{\chi}\right|\leq N}\left|\frac{\Gamma\left(\rho\right)\Gamma\left(\rho_{\chi}\right)}{\Gamma\left(\rho_{\chi}+\rho+1\right)}\right|
\]
\[
\ll N^{3/2}\log\left(N\right)^{2}\log\left(q\right)\log\left(qN\right)\log\left(N\right)^{1/4}
\]
\[
\ll N^{3/2}\log\left(N\right)^{13/4}\log\left(q\right).
\]
So, combining the pieces and taking $\theta=1/2-2\varepsilon$, we
finally obtain
\[
\sum_{1<q\leq N^{\theta}}\frac{1}{\varphi\left(q\right)}\left|\sum_{\chi\neq\chi_{0}}\overline{\chi}\left(N\right)\sum_{n\leq N\beta}\sum_{m\leq N\beta-n}\Lambda\left(n\right)\chi\left(n\right)\mu\left(m\right)\log\left(m\right)f\left(\frac{n+m}{N}\right)\right|
\]
\[
\ll_{f,\beta}N^{3/2}\log\left(N\right)^{17/4}\sum_{q\leq N^{1/2-2\varepsilon}}\log\left(q\right)+N^{3-\delta-2\varepsilon}\log\left(N\right)^{8}
\]
\[
\ll_{f,\beta,\varepsilon}N^{2-\varepsilon}
\]
and so the thesis.
\end{proof}

\section{examples}

In this section we propose some interesting examples that can be easily
deduced by the previous theorems, and that illustrate both the Sobolev
and the H\"older--Zygmund frameworks through standard compactly
supported weights.

We start with the classical Ces\`aro-Riesz weights, that is, the
function
\[
f_{CR,k}\left(x\right):=\begin{cases}
\left(1-x\right)^{k}, & x\in\left[0,1\right]\\
0, & \text{otherwise}
\end{cases}
\]
where $k\in\mathbb{R}_{0}^{+}$. If we assume that $k>1$, it is simple
to observe that $f_{CR,k}\left(x\right)$ verifies the hypotheses
of Theorem \ref{thm:MAINCGZ}, and so, making a simple change of variable,
we have immediately the following results.
\begin{thm}
Fix $0<\varepsilon<1/2$. Assume GRH and Conjecture \ref{conj:simple zeros}.
Let $N\geq4$ and $k>1$, $q\in\mathbb{N},1<q<N^{\theta},\,0<\theta<1$.
Then
\[
\frac{1}{\varphi\left(q\right)\Gamma\left(k+1\right)}\sum_{\chi\neq\chi_{0}}\overline{\chi}\left(N\right)\sum_{n\leq N}\left(1-\frac{n}{N}\right)^{k}\sum_{m\leq n}\Lambda\left(m\right)\chi\left(m\right)\mu\left(n-m\right)
\]
\[
=-\frac{1}{\varphi\left(q\right)}\sum_{\rho}\frac{N^{\rho}\Gamma\left(\rho\right)}{\zeta^{\prime}(\rho)}\sum_{\chi\neq\chi_{0}}\overline{\chi}\left(N\right)\sum_{\rho_{\chi}}\frac{N^{\rho_{\chi}}\Gamma\left(\rho_{\chi}\right)}{\Gamma\left(\rho_{\chi}+\rho+k+1\right)}
\]
\[
+O_{\varepsilon}\left(\frac{k\left(k-1\right)N^{1/2+\varepsilon}\log\left(q\right)^{3}\log\left(N\right)^{3}}{\Gamma\left(k+1\right)}\int_{0}^{1}\left(1-w\right)^{k-2}\left(1+w^{3/2+\varepsilon}\right)\left(1-\log\left(w\right)^{3}\right)dw\right)
\]
and the implicit constant depends only on $\varepsilon$. 
\end{thm}

Note that by Theorem \ref{thm:dobule series =00005Cchi} the double
series is absolutely convergent, and this is the reason of the form
of the previous theorem. Then, we can deduce the following theorem.
\begin{thm}
Assume all the hypotheses of the previous theorem and fix $0<\varepsilon<1/2$
and $\theta=1-2\varepsilon$. Then, the following weighted average
of diagonal twisted EH holds
\[
\sum_{1<q\leq N^{1-2\varepsilon}}\frac{1}{\varphi\left(q\right)\Gamma\left(k+1\right)}\left|\sum_{\chi\neq\chi_{0}}\overline{\chi}\left(N\right)\sum_{n<N}\left(1-\frac{n}{N}\right)^{k}\sum_{m\leq n}\Lambda\left(m\right)\chi\left(m\right)\mu\left(n-m\right)\right|
\]
\[
\ll_{\varepsilon}N^{2-\varepsilon}E\left(f^{\prime\prime}\right)
\]
where
\[
E\left(f^{\prime\prime}\right):=\frac{k\left(k-1\right)}{\Gamma\left(k+1\right)}\int_{0}^{1}\left(1-w\right)^{k-2}\left(1-\log\left(w\right)^{3}\right)\left(1+w^{2}\right)dw.
\]
and the implicit constant depends only on $\varepsilon$.
\end{thm}

Actually, since the error terms of the previous theorem depend explicitly
on $k$, the main result can be extended also to $k=1$. Now, take
\[
f_{Zyg}\left(x\right):=\begin{cases}
\left(1-x\right)\sin\left(\log\left(1-x\right)\right), & x\in[0,1)\\
0, & \text{otherwise;}
\end{cases}
\]
now, it is clear that, taking for instance $\alpha=0$ and $\beta=1+\varepsilon$
with $\varepsilon\in(0,1/2)$, then $f_{Zyg}\left(x\right)$ does
not fit the hypotheses of Theorem \ref{thm:MAINCGZ} but this function
is in $\mathcal{C}^{1}\left(\mathbb{R}\right).$ Indeed, 
\[
f_{Zyg}^{\prime}\left(x\right)=-\sin\left(\log\left(1-x\right)\right)-\cos\left(\log\left(1-x\right)\right)
\]
hence
\[
\left|f_{Zyg}^{\prime}\left(x\right)\right|\leq2
\]
and, by continuity, $f_{Zyg}\left(0\right)=f_{Zyg}\left(1\right)=0$,
so $f_{Zyg}\left(x\right)\in\text{Lip}\left(\mathbb{R}\right)\subset\mathcal{C}^{1}\left(\mathbb{R}\right)$. 

Observe that this function is a variant of the Ces\`aro-Riesz weight
of order $1$. Then, for this function, we can prove what follows.
\begin{thm}
Fix $0<\varepsilon<1/4.$ Assume GRH, let $N\geq4$, $T,U\geq4$,
$q\in\mathbb{N},\,1<q<N^{\theta},\,0<\theta<1$. Assume also the Conjecture
\ref{conj:simple zeros}. Then
\[
\frac{1}{\varphi\left(q\right)}\sum_{\chi\neq\chi_{0}}\overline{\chi}\left(N\right)\sum_{n\leq N}\sum_{m\leq N-n}\Lambda\left(n\right)\chi\left(n\right)\mu\left(m\right)f_{Zyg}\left(\frac{n+m}{N}\right)
\]
\[
=-\frac{1}{\varphi\left(q\right)}\sum_{\rho:\left|\gamma\right|\leq U}\frac{\Gamma\left(\rho\right)}{\zeta^{\prime}(\rho)}\sum_{\chi\neq\chi_{0}}\overline{\chi}\left(N\right)\sum_{\rho_{\chi}:\left|\gamma_{\chi}\right|\leq T}\frac{\Gamma\left(\rho_{\chi}\right)}{\Gamma\left(\rho_{\chi}+\rho+2\right)}\sum_{k\leq N-1}\Delta_{1/N}^{2}\left(f_{Zyg},\frac{k}{N}\right)k^{\rho+\rho_{\chi}+1}
\]
\[
+O_{f_{Zyg},\varepsilon,\theta}\left(N^{3/2}\log\left(qT\right)^{2}+\frac{N^{2}\log\left(qNT\right)^{2}}{T}\right.
\]
\[
\left.+\log\left(q\right)^{3}N^{3/2+\varepsilon}\log\left(N\right)^{3}\left[1+\frac{N\log\left(qT\right)^{2}}{T}\right]\left[1+\frac{N^{1-\varepsilon}}{U^{1-\varepsilon}}\right]\right).
\]
\end{thm}

From the previous theorem, we derive the following upper bound.
\begin{thm}
Assume all the hypotheses of the previous theorem and fix $0<\varepsilon<1/4$
and $\theta=1-2\varepsilon$. We get
\[
\sum_{\underset{{\scriptstyle \left(N,q\right)=1}}{1<q\leq N^{1/2-2\varepsilon}}}\frac{1}{\varphi\left(q\right)}\left|\sum_{\chi\neq\chi_{0}}\overline{\chi}\left(N\right)\sum_{n\leq N}\sum_{m\leq N-n}\Lambda\left(n\right)\chi\left(n\right)\mu\left(m\right)f_{Zyg}\left(\frac{n+m}{N}\right)\right|\ll_{f_{Zyg},\varepsilon}N^{2-\varepsilon}.
\]
\end{thm}

\section*{Acknowledgements}

The author is member of the Gruppo Nazionale per l\textquoteright Analisi
Matematica, la Probabilit\`a e le loro Applicazioni (GNAMPA) of the
Istituto Nazionale di Alta Matematica (INdAM). The author thanks Professor Alessandro Zaccagnini and Professor Alessandro Gambini for some discussions on this topic.

\subsection*{Conflict of interest} 
The author declare that he has no conflict of interest.

\end{document}